\documentclass[12pt]{amsart}
\usepackage{amsmath,amsthm,amsfonts,amssymb} 
\usepackage[all]{xy}
\usepackage{color}

\usepackage{bbm}
\usepackage{latexsym}
\usepackage{amscd}
\usepackage{comment}

\usepackage[margin=1in]{geometry}      
\allowdisplaybreaks 

\numberwithin{equation}{subsection} 
\newtheorem{thm}[equation]{Theorem} 
\newtheorem*{thmA}{Theorem} 

\newtheorem{prop}[equation]{Proposition}
\newtheorem{lemma}[equation]{Lemma} 

\newtheorem{example}[equation]{Example}
\newtheorem{remark}[equation]{Remark}

\newtheorem{definition}[equation]{Definition}

\DeclareMathAlphabet{\mathpzc}{OT1}{pzc}{m}{it} 
\DeclareMathOperator{\gr}{gr} 
\DeclareMathOperator{\kG}{\mathbbm{k}G}

\DeclareMathOperator{\Tot}{Tot}

\DeclareMathOperator{\coh}{H}

\DeclareMathOperator{\Ext}{Ext}
\DeclareMathOperator{\Hom}{Hom}
\DeclareMathOperator{\ima}{Im}
\DeclareMathOperator{\Ker}{Ker}
\DeclareMathOperator{\Span}{Span}

\renewcommand{\k}{{\mathbbm k}}
\newcommand*{\grH}{\ensuremath{\text{\upshape gr}}\, H}

\newcommand{\Z}{{\mathbb Z}}
\newcommand{\N}{{\mathbb N}}

\newcommand{\ot}{\otimes}

\newcommand*{\YD}{\,^{G}_{G} \mathcal{YD}}

\newcommand{\DOT}{\setlength{\unitlength}{1pt}\begin{picture}(2.5,2)
               (1,1)\put(2,2.5){\circle*{2}}\end{picture}}

\newcommand{\bu}{\DOT}

\begin{document}



\thanks{The third author was partially supported by NSF grant DMS-1401016.}

\title[Cohomology of pointed Hopf algebras]
{Finite generation of some cohomology rings \\  
via twisted tensor product and Anick resolutions}

\author{Van C.\ Nguyen} 
\address{Department of Mathematics\\Hood College\\Frederick, MD 21701}
\email{nguyen@hood.edu}

\author{Xingting Wang}
\address{Department of Mathematics\\Temple University \\Philadelphia, PA 19122}
\email{xingting@temple.edu}

\author{Sarah Witherspoon}
\address{Department of Mathematics\\Texas A\&M University \\College Station, TX 77843}
\email{sjw@math.tamu.edu}

\date{October 18, 2017}

\begin{abstract}
Over a field of prime characteristic $p>2$, we prove that the 
cohomology rings of some pointed Hopf algebras of 
dimension $p^3$ are finitely generated.
These are Hopf algebras arising in the ongoing classification of finite
dimensional Hopf algebras in positive characteristic, and include
bosonizations of Nichols algebras of Jordan type in a general setting as well as
their liftings when $p=3$. 
Our techniques are applications of twisted tensor product
resolutions and Anick resolutions in combination with May spectral sequences.
\end{abstract}

\maketitle
\tableofcontents


\section{Introduction}

The cohomology ring of a finite dimensional Hopf algebra is conjectured to
be finitely generated.
Friedlander and Suslin~\cite{FS} proved this for cocommutative Hopf algebras,
generalizing earlier results of Evens~\cite{Evens}, 
Golod~\cite{golod}, and Venkov~\cite{venkov}
for finite group algebras
and of Friedlander and Parshall~\cite{FP} for restricted Lie algebras. 
There are many finite generation results as well 
for various types of noncocommutative Hopf algebras
(see, e.g.,~\cite{BNPP,ginzburg-kumar93,Gordon,MPSW,SV}).
Most of these results are in characteristic~0. 
In this paper, we prove finite generation for classes of
noncocommutative Hopf algebras in prime characteristic~$p>2$.
These are some of the pointed Hopf algebras arising in classification work
of the first two authors.

Our main result is the following combination of
Theorems \ref{thm:main1} and \ref{thm:main2} below:

\begin{thmA}
Let $\k$ be an algebraically closed field of prime characteristic $p>2$. Consider the following Hopf algebras over $\k$:
 \begin{enumerate}
  \item the $p^2q$-dimensional bosonization $R \# \kG$ of a rank two Nichols algebra $R$ of Jordan type over a cyclic group $G$ of order $q$, where $q$ is divisible by $p$; and
  \item a lifting $H$ of $R\#\kG$ when $p=q=3$.
 \end{enumerate}
Then the cohomology rings of $R \# \kG$ and of $H$ are finitely generated. 
\end{thmA}

Our theorem is exclusively an odd characteristic result since the Nichols
algebra of Jordan type does not appear in characteristic~2.
Instead there is another related Nichols algebra~\cite{CLW} that
will require different techniques.
Part (2) of our main theorem above is only stated for
characteristic~3; this is because we use the classification of such
Hopf algebras from~\cite{NWa}, and liftings are only known completely
in this case.
We expect our homological techniques will be able to handle liftings
when $p>3$ once more is known about their structure.

More specifically, we let $R\#\kG$ be a $p^2q$-dimensional Hopf algebra given by
the bosonization of a Jordan plane as introduced in~\cite{CLW}
(see Section~\ref{subsec:settingR} below).
We prove in Theorem~\ref{thm:main1} that the cohomology ring
of such a Hopf algebra $R\# \k G$,
that is $\coh^*(R\# \k G,\k):= \Ext^*_{R\# \k G}(\k ,\k )$, is finitely generated.
We also consider liftings $H$ of $R\# \k G$ in
the special case $p=q=3$ (see Section~\ref{subsec:settingH} below).
We prove in Theorem~\ref{thm:main2} that the cohomology ring
$\coh^*( H, \k)$ is finitely generated. 

Our techniques for proving the main theorem above rely on the
May spectral sequence for the cohomology of a filtered algebra.
In either setting, (1) or (2), we may choose a filtration for which the
associated graded algebra is a truncated polynomial ring whose 
cohomology is straightforward.
The hard work is in finding some permanent cycles as required to use a
spectral sequence lemma that goes back to Friedlander
and Suslin~\cite{FS}.
This we do in two ways, by constructing two types of resolutions which
should be of independent interest.
The general definitions of the resolutions are not new, but
we provide here some nontrivial examples, with our main
theorem as an application.
The first resolution is the twisted tensor product resolution of~\cite{SW},
recalled in Section~\ref{subsec:twisted}, used in
Section~\ref{subsec:resoln R}, and iterated in 
Section~\ref{subsec:resoln bosonization}, 
to obtain a resolution over the bosonization $R\# \k G$ 
of the Nichols algebra $R$. 
This is used in Section~\ref{subsec:cohomology grH} 
to prove finite generation of its cohomology. 
The second resolution is the Anick resolution~\cite{Anick}, 
explained in Section~\ref{subsec:Anick}, with a general result 
for the Anick resolution of a truncated polynomial ring 
in Section~\ref{subsec:tp}.
This Anick resolution is used in Section~\ref{subsec:cohomology lifting} 
to prove finite generation of the cohomology of the liftings $H$. 
We could have chosen to work with just one type of resolution,
either Anick  or  twisted tensor product,
for proofs of both parts~(1) and~(2) of our main theorem above.
We instead chose to work with both resolutions to illustrate
a wider variety of techniques available, each having its own advantages.


\section{Settings}

Throughout, let $\k$ be a field.
For our main results we will assume that $\k$ is algebraically closed field
of prime characteristic $p>2$ since we work with Hopf algebras found
in classification work under this assumption.
The tensor product $\ot$ is $\ot_\k$ unless specified otherwise. 
In this section we will define the Nichols algebras and pointed Hopf algebras 
that are featured in this paper, and summarize some structural results that
will be needed.


\subsection{Pointed Hopf algebras}

Let $H$  
be any finite dimensional pointed Hopf algebra over $\k$. The \emph{coradical} (the sum of all simple subcoalgebras) of $H$ is $H_0 = \kG$, a Hopf subalgebra of $H$ generated by the grouplike elements $G:=\{ g \in H \,|\, \Delta(g)=g \otimes g\}$,
where $\Delta$ is the coproduct on $H$. 

Let 
$$H_0 \subseteq H_1 \subseteq H_2  \subseteq \cdots \subseteq H $$ 
be the \emph{coradical filtration} of $H$, where $H_n=\Delta^{-1}(H\otimes H_{n-1} + H_0\otimes H)$ inductively, see \cite[Chapter 5]{MO93}. Consider the associated graded Hopf algebra $\grH = \bigoplus_{n \geq 0} H_n/H_{n-1}$, with the convention $H_{-1}=0$. Note that the zero term of $\grH$ equals its coradical, i.e., $(\grH)_0=H_0$. There is a projection $\pi: \grH \rightarrow H_0$ and an inclusion $\iota: H_0 \rightarrow \grH$ such that $\pi \iota = 1$ (the identity map on $H_0$). Let $R$ be the algebra of coinvariants of $\pi$:
$$R:=(\grH)^{co\, \pi} = \{ h \in \grH \,: \, ({\mathbf{1}_H} \otimes \pi)\Delta(h)= h \otimes 1\} , $$
where ${\mathbf{1}_H}$ denotes the identity map on $H$.
By results of Radford \cite{R85} and Majid \cite{Mj}, $R$ is a Hopf algebra in the braided category $\YD$ of left Yetter-Drinfeld modules over $H_0=\kG$. Moreover, $\grH$ is the \emph{bosonization} (or \emph{Radford biproduct}) of $R$ and $H_0$ so that $\grH \cong R \# H_0$ with the Hopf structure given in \cite[Theorem 10.6.5]{MO93}.
As an algebra, it is simply the smash product of $H_0$ with $R$,
analogous to a semidirect product of groups.


\subsection{Our setting: Rank two Nichols algebra and its bosonization}
\label{subsec:settingR}

For Sections~\ref{sec:twisted} and~\ref{subsec:cohomology grH}, we use the following setup. 
Let $\k$ be a field of prime characteristic $p>2$. 
Let $G:=\langle g\rangle \cong \mathbb{Z}/q\mathbb{Z}$ be a cyclic group whose order $q$ is divisible by $p$. Consider
\[ 
   \displaystyle R := \k\langle x,y \rangle \Big{/} \left(x^p, \ y^p , \ yx-xy-\frac{1}{2}x^2 \right),
\]
which is a $p^2$-dimensional algebra as described in \cite[Theorem 3.5]{CLW}, with a vector space basis $\{x^iy^j \mid 0 \leq i,j \leq p-1\}$. Observe that $R$ is the Nichols algebra of a rank two Yetter-Drinfeld module $V=\k x+ \k y$ over $G$, where the $G$-action on $V$ is given by
\[
    {}^g x = x \ \ \ \text{ and } \ \ \ {}^g y = x+y,
\]
(here the left superscript indicates group action), and the $G$-gradings of $x$ and $y$ are both given by $g$. 

Let $R\# \kG$ be the bosonization of $R$ and $\kG$. It is the corresponding $p^2q$-dimensional pointed Hopf algebra as studied in~\cite[Corollary 3.14]{CLW}, \cite[\S 3 (Case B)]{NWa}, and \cite[\S 4]{NWi}. 
Its Hopf structure is given by: 
\begin{align*}
  \Delta(x) &= x\ot 1 + g\ot x, & \Delta(y) &= y\ot 1 + g\ot y, & \Delta(g) &= g\ot g, \\
  \varepsilon(x) &= 0, & \varepsilon(y) &=0, & \varepsilon(g)&=1, \\
   S(x) &= -g^{-1}x, & S(y) &= -g^{-1}y, & S(g)&=g^{-1},
\end{align*}
where $\Delta$ is the coproduct, $\varepsilon$ is the counit, and $S$ is the antipode map of $R\# \kG$.

\begin{remark} We remark the following:
\begin{enumerate}
 \item These $R, G, R\# \kG$ appear in \cite{CLW,NWa,NWi} for various purposes. In their settings, $G$ is a cyclic group of order $p$. Here, we consider a more general setting with $G$ being cyclic of order $q$ divisible by $p$. 
 
 \item In \cite{NWa}, the first two authors classified $p^3$-dimensional pointed Hopf algebra over prime characteristic $p$. In their classification work, this $p^2$-dimensional Nichols algebra $R$ of Jordan type is unique, up to isomorphism, and only occurs when $p>2$. 
 
 \item The authors of~\cite{CLW} used right modules in their settings, whereas left modules were used in the classification work in \cite{NWa}, thus inducing a sign difference in the relation $yx-xy-\frac{1}{2}x^2$ of $R$ there. 
Here, we adopt the relation described in \cite{CLW}. In Section~\ref{subsec:cohomology lifting}, when we study the cohomology of pointed Hopf algebras $H$ lifted from the associated graded algebras $\grH \cong R \# \kG$, we modify the relations of the lifting structure given in \cite{NWa} accordingly, see the next Section~\ref{subsec:settingH}. 
\end{enumerate}
\end{remark}


\subsection{Our setting: A class of $27$-dimensional pointed Hopf algebras}
\label{subsec:settingH}

Let $\k$ be a field of characteristic~$p=3$ and consider Hopf algebras $H(\epsilon,\mu,\tau)$ defined from three scalar parameters $\epsilon,\mu,\tau$ as in~\cite{NWa}.
These are pointed Hopf algebras of dimension $27$ whose associated graded algebra
is $\gr H(\epsilon,\mu,\tau) \cong R \# \kG$, the
Hopf algebra described in Section~\ref{subsec:settingR}, in this case $p=q=3$.
 
As an algebra, $H(\epsilon,\mu,\tau)$ is generated by $g,x,y$ with relations 
\begin{gather*}
\qquad \qquad \ \ g^3=1, \qquad x^3=\epsilon x, \qquad y^3= -\epsilon y^2 - (\mu\epsilon -\tau -\mu^2)y, \\
yg-gy= xg + \mu(g-g^2), \qquad  xg-gx=-\epsilon (g-g^2), \\
yx-xy= - x^2 + (\mu+\epsilon)x +\epsilon y + \tau(1-g^2), 
\end{gather*}
where $\epsilon\in \{0,1\}$ and $\tau,\mu \in \k$ are arbitrary scalars.
The coalgebra structure is the same as that of $R\# \k G$ described 
in Section~\ref{subsec:settingR}.  
By setting $w = g-1$,
we get a new presentation in which the 
generators are $w,x,y$ and the relations are:
\begin{gather*}
w^3=0, \qquad x^3=\epsilon x, \qquad y^3= -\epsilon y^2 - (\mu\epsilon -\tau -\mu^2)y, \\
yw-wy= wx+x- (\mu - \epsilon) (w^2+w), \qquad  xw-wx=\epsilon (w^2+w), \\
yx-xy= - x^2 + (\mu+\epsilon)x +\epsilon y -\tau(w^2-w). 
\end{gather*}
This choice of generating set will be convenient for our
homological arguments later. 
As shown in~\cite{NWa}, $H(\epsilon,\mu,\tau)$
has dimension~27, with vector space basis 
$\{w^i x^j y^k \mid 0\leq i,j,k\leq 2\}$.


\section{Twisted tensor product resolutions}
\label{sec:twisted}

In this section, we apply the construction of twisted tensor product resolutions introduced in \cite{SW} to our Nichols algebra $R$ and its bosonization $R\# \kG$ that was defined in Section~\ref{subsec:settingR}.


\subsection{The resolution construction}
\label{subsec:twisted}

Let $A$ and $B$ be associative algebras over $\k$ with multiplication maps $m_A : A \ot A \rightarrow A$ and $m_B : B \ot B \rightarrow B$, and multiplicative identities $1_A$ and $1_B$, respectively. We write $\mathbf{1}$ for the identity map on any set. 

A \emph{twisting map} $\tau: B \ot A \rightarrow A \ot B$ is a bijective $\k$-linear map for which $\tau(1_B \ot a) = a \ot 1_B$ and $\tau(b \ot 1_A) = 1_A \ot b$, for all $a \in A$ and $b \in B$, and 
\begin{equation}\label{eqn:twist-def}
\tau \circ (m_B \ot m_A) = (m_A \ot m_B) \circ (\mathbf{1} \ot \tau \ot \mathbf{1}) \circ (\tau \ot \tau) \circ (\mathbf{1} \ot \tau \ot \mathbf{1})
\end{equation}
as maps $B \ot B \ot A \ot A \rightarrow A \ot B$. The \emph{twisted tensor product algebra} $A \ot_\tau B$ is the vector space $A \ot B$ together with multiplication $m_\tau$ given by such a twisting map $\tau$, that is, 
$m_{\tau} : (A\ot B)\ot (A\ot B)\rightarrow A\ot B$
is given by $m_{\tau} = (m_A\ot m_B)\circ (\mathbf{1} \ot \tau \ot \mathbf{1})$. 

\begin{definition}\cite[Definition 5.1]{SW}
\label{def:compatible twisting}{\em 
Let $M$ be an $A$-module with module structure map $\rho_{A,M} : A \ot M \rightarrow M$. 
We say $M$ is \emph{compatible with the twisting map $\tau$} if there is a bijective $\k$-linear map $\tau_{B,M}: B \ot M \rightarrow M \ot B$ such that
\begin{align} \label{compatible twisting1}
\tau_{B,M} \circ (m_B \ot \mathbf{1}) &= (\mathbf{1} \ot m_B) \circ (\tau_{B,M} \ot \mathbf{1}) \circ (\mathbf{1} \ot \tau_{B,M}), \text{ and} \\
\tau_{B,M} \circ (\mathbf{1} \ot \rho_{A,M}) &= (\rho_{A,M} \ot \mathbf{1}) \circ (\mathbf{1} \ot \tau_{B,M}) \circ (\tau \ot \mathbf{1}) \label{compatible twisting2}
\end{align}
as maps on $B \ot B \ot M$ and on $B \ot A \ot M$, respectively.}
\end{definition}

Let $M$ be an $A$-module that is compatible with $\tau$. We say a projective $A$-module resolution $P_{\bu} (M)$ of $M$ is \emph{compatible with the twisting map $\tau$} if each module $P_i(M)$ is compatible with $\tau$ via maps $\tau_{B,i}$ for which $\tau_{B,\bu}: B \ot P_{\bu} (M) \rightarrow P_{\bu} (M) \ot B$ is a $\k$-linear chain map lifting $\tau_{B,M}: B \ot M \rightarrow M \ot B$. Let $N$ be a $B$-module and let $P_{\bu}(N)$ be a projective resolution of $N$ over $B$. We put an $A \ot_{\tau} B$-module structure on the bicomplex $P_{\bu}(M) \ot P_{\bu}(N)$ by using maps $\tau_{B,\bu}$.

Under such compatibility conditions, the twisted tensor product resolutions for left modules over $A \ot_\tau B$ were constructed in \cite{SW} satisfying the following theorem. 

\begin{thm}\cite[Theorem 5.12]{SW}
\label{thm:twisted resoln}
Let $A$ and $B$ be $\k$-algebras with twisting map $\tau: B \ot A \rightarrow A \ot B$. Let $P_{\bu}(M)$ be an $A$-projective resolution of $M$ and $P_{\bu}(N)$ be a $B$-projective resolution of $N$. Assume 
\begin{enumerate}
 \item[(a)] $M$ and $P_{\bu}(M)$ are compatible with $\tau$, and 
 \item[(b)] $Y_{i,j}=P_i(M) \ot P_j(N)$ is a projective $A \ot_{\tau} B$-module, for all $i,j$. 
\end{enumerate} 
Then the twisted tensor product complex $Y_{\bu}=\Tot(P_{\bu}(M) \ot P_{\bu}(N))$ with
$$Y_n = \bigoplus_{i+j=n} P_i(M) \ot P_j(N)$$
is a projective resolution of $M \ot N$ as a module over the twisted tensor product $A \ot_{\tau} B$. 
\end{thm}

By $\Tot (P_{\bu}(M) \ot P_{\bu}(N))$ we mean the total complex of
the bicomplex $P_{\bu}(M) \ot P_{\bu}(N)$, that is the complex whose
$n$th component is $Y_n = \oplus_{i+j=n} (P_i(M)\otimes P_j(N))$ and
differential is $d_n = \sum_{i+j=n} d_{ij}$ where
$d_{ij}= d_i\ot {\mathbf{1}} + (-1)^i{\mathbf{1}}\ot d_j$. 

In the statement of \cite[Theorem 5.12]{SW}, we have replaced one
of the hypotheses by our hypothesis (b) above,
and in this case the proof is given by~\cite[Lemmas~5.8 and~5.9]{SW}.
In some contexts, such as ours here, the hypothesis (b) can be
checked directly, in which case this version of the theorem is 
sufficient.

For the rest of this section, let $R,G$ be as defined in Section~\ref{subsec:settingR}. 
Let $\k$ also denote the trivial (left) $R$- or $(R \# \kG)$-module, that is, the field $\k$
with action given by the augmentation
$\varepsilon(x)=0$, $\varepsilon(y)=0$, $\varepsilon(g)=1$.
We will construct a resolution of $\k$ as an $R$-module via
a twisted tensor product, and as an $(R \# \kG)$-module by iterating
the twisted product construction.


\subsection{Resolution over the Nichols algebra $R$}
\label{subsec:resoln R}

Let $\k$ be a field of characteristic $p>2$, and let $R$ be
the Nichols algebra described in Section~\ref{subsec:settingR}.
Let $\k$ be the $R$-module on which $x$ and $y$ both act as~0. 

We start with a construction of a resolution of $\k$ as an $R$-module.
Using the relation (3.9) in \cite[Lemma 3.8]{CLW}, we can view 
$R$ as the twisted tensor product 
$$R=\k\langle x,y \rangle \Big{/} \left(x^p, \ y^p , \ yx-xy-\frac{1}{2}x^2 \right) \cong A\ot_{\tau} B,$$ 
where
$A:=\k[x]/(x^p)$, $B:=\k[y]/(y^p)$. The twisting map 
$\tau: B\ot A \rightarrow A\ot B$ is defined by 
\[ \tau(y^r\ot x^\ell) = \sum_{t=0}^r {r \choose t} \left(\frac{1}{2}\right)^t [\ell]^{[t]} \ x^{\ell+t} \ot y^{r-t},\] 
where we use the convention: 
$$[\ell]^{[t]}=\ell(\ell+1)(\ell+2) \cdots (\ell+t-1),$$ 
with $[\ell]^{[0]}=1$, for any $\ell$. 

Consider the following free resolutions of $\k$ as $A$-module
and as $B$-module, respectively: 
\[
\begin{xy}*!C\xybox{
\xymatrix{
   P_{\bu}(A): & \cdots\ar[r]^{x^{p-1}\cdot} & A\ar[r]^{x\cdot}
                 & A \ar[r]^{x^{p-1}\cdot}  & A\ar[r]^{x\cdot}
     & A\ar[r]^{\varepsilon} & \k\ar[r] & 0 \\
  P_{\bu}(B): & \cdots\ar[r]^{y^{p-1}\cdot} & B\ar[r]^{y\cdot}
         & B \ar[r]^{y^{p-1}\cdot} & B \ar[r]^{y\cdot}
     & B \ar[r]^{\varepsilon} & \k\ar[r] & 0. 
}}
\end{xy}
\]
The map $\varepsilon$ on $A$ (respectively, on $B$) 
takes $x$ to 0 (respectively, $y$ to 0). 
Consider the tensor product $P_{\bu}(A)\ot P_{\bu}(B)$ as a graded vector
space. The total complex is a complex of vector spaces with differential
in degree $n$ given by 
\[
   d_n = \sum_{i+j=n} (d_i\ot 1 + (-1)^i \ot d_j) .
\]
We will put the structure of an $R$-module on each $P_i(A)\ot P_j(B)$
so that this tensor product complex is an $R$-projective resolution
of $\k$ as an $R$-module. 
This will follow from~\cite[Lemmas~5.8 and~5.9]{SW} once we define
a chain map as in Definition~\ref{def:compatible twisting} and check that the
resulting $R$-modules are indeed projective.
That is, as in Definition~\ref{def:compatible twisting}, there are bijective
$\k$-linear maps $\tau_{B,i}$ that we abbreviate here as $\tau_i$,
for which the diagram
\[{\small
\begin{xy}*!C\xybox{
\xymatrixcolsep{3.3pc}\xymatrix{
  \cdots \ar[r]^{\mathbf{1} \ot x^{p-1} \cdot} & B\ot A\ar[d]^{\tau_3}\ar[r]^{\mathbf{1} \ot x \cdot} 
   & B\ot A \ar[d]^{\tau_2}\ar[r]^{\mathbf{1} \ot x^{p-1} \cdot} & B\ot A\ar[d]^{\tau_1}
     \ar[r]^{\mathbf{1}\ot x \cdot}
    & B\ot A \ar[d]^{\tau_0}\ar[r]^{\mathbf{1} \ot \varepsilon} & B\ot \k\ar[r]\ar[d]^{\cong} 
   & 0 \\
  \cdots \ar[r]^{x^{p-1} \cdot \ot \mathbf{1}} & A\ot B\ar[r]^{x \cdot \ot \mathbf{1}} 
   & A\ot B\ar[r]^{x^{p-1}\cdot \ot \mathbf{1}} & A\ot B\ar[r]^{x \cdot \ot \mathbf{1}}
   & A\ot B \ar[r]^{\varepsilon\ot \mathbf{1}} & \k\ot B \ar[r] & 0 
}}
\end{xy}}
\]
commutes and conditions \eqref{compatible twisting1} and \eqref{compatible twisting2} hold.
We claim that the following maps $\tau_i$ satisfy the above conditions.

\begin{lemma}
For any integer $i \geq 0$, let $\tau_i: B \ot A \rightarrow A \ot B$ be defined as follows:
\[
\tau_i(y^r \ot x^\ell) = \begin{cases} \tau(y^r \ot x^\ell), \quad & i \text{ is even} \vspace{0.7em} \\
\sum_{t=0}^r {r \choose t} \left(\frac{1}{2}\right)^t [\ell+1]^{[t]} \ x^{\ell+t} \ot y^{r-t}, \quad & i \text{ is odd}.
\end{cases}
\]
Then 
\begin{enumerate}
 \item[(a)] $\tau_i$ is a bijective $\k$-linear map whose inverse is
\[
\tau_i^{-1}(x^\ell \ot y^r) = \begin{cases} \sum_{t=0}^r {r \choose t} \left(- \frac{1}{2}\right)^t [\ell]^{[t]} \ y^{r-t} \ot x^{\ell+t}, \quad & i \text{ is even} \vspace{0.7em} \\
\sum_{t=0}^r {r \choose t} \left(- \frac{1}{2}\right)^t [\ell+1]^{[t]} \ y^{r-t} \ot x^{\ell+t}, \quad & i \text{ is odd}.
\end{cases}
\]

 \item[(b)] $\tau_i$ satisfies conditions \eqref{compatible twisting1} and \eqref{compatible twisting2}. In particular, 
\begin{align*}
\tau_i \circ (m_B \ot \mathbf{1}) &= (\mathbf{1} \ot m_B) \circ (\tau_i \ot \mathbf{1}) \circ (\mathbf{1} \ot \tau_i) \text{ and } \\
\tau_i \circ (\mathbf{1} \ot m_A) &= (m_A \ot \mathbf{1}) \circ (\mathbf{1} \ot \tau_i) \circ (\tau \ot \mathbf{1}),
\end{align*}
as maps on $B \ot B \ot A$ and on $B \ot A \ot A$, respectively. 

 \item[(c)] Each square in the above diagram commutes. 
\end{enumerate}
\noindent
Consequently, 
$\k$ and its resolution $P_{\bu}(A)$ are compatible with $\tau$. 
\end{lemma}

\begin{proof} 
Let $\tau_i$ be defined as in the lemma.\medskip

\noindent
(a) \textbf{Bijection:} To show $\tau_i$ is bijective with the given inverse, it suffices to show that $\tau_i \circ \tau_i^{-1}= \mathbf{1}_{A \ot B}$ and $\tau_i^{-1} \circ \tau_i = \mathbf{1}_{B \ot A}$. Here, we will check $\tau_i \circ \tau_i^{-1}= \mathbf{1}_{A \ot B}$ for the case when $i$ is odd; the remaining case is similar. 
\begin{align*}
\tau_i \circ \tau_i^{-1}(x^\ell \ot y^r) &= \tau_i \left( \sum_{k=0}^r {r \choose k} \left(- \frac{1}{2}\right)^k [\ell+1]^{[k]} \ y^{r-k} \ot x^{\ell+k} \right) \\
&= \sum_{k=0}^r {r \choose k} \left(- \frac{1}{2}\right)^k [\ell+1]^{[k]} \sum_{s=0}^{r-k} {{r-k} \choose s} \left(\frac{1}{2}\right)^s [\ell+k+1]^{[s]} \ x^{\ell+k+s} \ot y^{r-k-s} \\
&= \sum_{k=0}^r \sum_{s=0}^{r-k} {r \choose k} {{r-k} \choose s} \left(- \frac{1}{2}\right)^k \left(\frac{1}{2}\right)^s [\ell+1]^{[k]}  [\ell+k+1]^{[s]} \ x^{\ell+k+s} \ot y^{r-k-s} \\
&= \sum_{t=0}^r \left( \sum_{k+s=t} {r \choose k} {{r-k} \choose s} (-1)^k \left(\frac{1}{2}\right)^{k+s} [\ell+1]^{[k]}  [\ell+k+1]^{[s]} \right)\ x^{\ell+t} \ot y^{r-t}.
\end{align*}
It suffices to show 
\[
 \sum_{k+s=t} {r \choose k} {{r-k} \choose s} (-1)^k \left(\frac{1}{2}\right)^{k+s}  [\ell+1]^{[k]}  [\ell+k+1]^{[s]} = \begin{cases}1, & \quad \text{ if } t=0 \\ 0, & \quad \text{ if } t \neq 0. \end{cases}
\]

The case $t=0$ is clear, as $k=s=0$, all coefficients become $1$. For $t > 0$, 
\begin{align*}
\sum_{k=0}^t &{r \choose k} {{r-k} \choose t-k} (-1)^k \left(\frac{1}{2}\right)^t  [\ell+1]^{[k]}  [\ell+k+1]^{[t-k]} \\
&= \frac{(\ell+t)!}{\ell!} \left(\frac{1}{2}\right)^t \left( \sum_{k=0}^t (-1)^k {r \choose k} {{r-k} \choose t-k}   \right) \\
&= \frac{(\ell+t)!}{\ell!} \left(\frac{1}{2}\right)^t \left( \sum_{k=0}^t (-1)^k {r \choose t} {t \choose k} \right) \\
&= \frac{(\ell+t)!}{\ell!} \left(\frac{1}{2}\right)^t {r \choose t} \left( \sum_{k=0}^t (-1)^k {t \choose k} \right) = 0. 
\end{align*}
Therefore, $\tau_i \circ \tau_i^{-1}= \mathbf{1}_{A \ot B}$. Similarly,  $\tau_i^{-1} \circ \tau_i = \mathbf{1}_{B \ot A}$ and hence $\tau_i$ is bijective.\medskip

\noindent
(b) \textbf{Compatible conditions:}
To show that the maps $\tau_i$ satisfy conditions \eqref{compatible twisting1} and \eqref{compatible twisting2}, observe that if $i$ is $0$ or even, then both conditions hold from the definition of $\tau$, as $\tau_i=\tau$ in this case and satisfies (\ref{eqn:twist-def}). It remains to check the case when $i$ is odd. 

Let us verify \eqref{compatible twisting1} first. On $B \ot B \ot A$, the left hand side is
\[
\tau_i \circ (m_B \ot \mathbf{1})(y^{r_1} \ot y^{r_2} \ot x^\ell) = \tau_i(y^{r_1+r_2} \ot x^\ell) = \sum_{t=0}^{r_1+r_2} {{r_1+r_2} \choose t} \left(\frac{1}{2}\right)^t [\ell+1]^{[t]} \ x^{\ell+t} \ot y^{r_1+r_2-t},
\]
while the right hand side is
\begin{align*}
&(\mathbf{1} \ot m_B) \circ (\tau_i \ot \mathbf{1}) \circ (\mathbf{1} \ot \tau_i)(y^{r_1} \ot y^{r_2} \ot x^\ell) \\
&= (\mathbf{1} \ot m_B) \circ (\tau_i \ot \mathbf{1})\left(y^{r_1} \ot \sum_{k=0}^{r_2} {{r_2} \choose k} \left(\frac{1}{2}\right)^k [\ell+1]^{[k]} \ x^{\ell+k} \ot y^{r_2-k} \right) \\
&=(\mathbf{1} \ot m_B)\left[\sum_{k=0}^{r_2} {{r_2} \choose k} \left(\frac{1}{2}\right)^k [\ell+1]^{[k]} 
\left(\sum_{s=0}^{r_1} {{r_1} \choose s} \left(\frac{1}{2}\right)^s [\ell+k+1]^{[s]} \ x^{\ell+k+s} \ot y^{r_1-s} \right) \ot y^{r_2-k} \right] \\
&=\sum_{t=0}^{r_1 + r_2} \left( \sum_{k+s=t} {{r_2} \choose k} {{r_1} \choose s} \left(\frac{1}{2}\right)^t [\ell+1]^{[k]} \, [\ell+k+1]^{[s]} \right)\ x^{\ell+t} \ot y^{r_1+r_2-t}. 
\end{align*}
It is straightforward to check that $[\ell+1]^{[k]} \,  [\ell+k+1]^{[s]} = [\ell+1]^{[k+s]}$, and 
$$\sum_{k=0}^t {r_2 \choose k} {{r_1} \choose {t-k}} = {{r_1+r_2} \choose t}.$$
This gives us the equality \eqref{compatible twisting1} as desired. 

For \eqref{compatible twisting2}, on $B \ot A \ot A$, the left hand side is
\[
\tau_i \circ (\mathbf{1} \ot m_A)(y^r \ot x^{\ell_1} \ot x^{\ell_2}) = \tau_i(y^r \ot x^{\ell_1+\ell_2}) = \sum_{t=0}^r {r \choose t} \left(\frac{1}{2}\right)^t [\ell_1+\ell_2+1]^{[t]} \ x^{\ell_1+\ell_2+t} \ot y^{r-t},
\]
while the right hand side is
\begin{align*}
&(m_A \ot \mathbf{1}) \circ (\mathbf{1} \ot \tau_i) \circ (\tau \ot \mathbf{1})(y^r \ot x^{\ell_1} \ot x^{\ell_2}) \\
&= (m_A \ot \mathbf{1}) \circ (\mathbf{1} \ot \tau_i) \left[ \left( \sum_{k=0}^r {r \choose k} \left(\frac{1}{2}\right)^k [\ell_1]^{[k]} \ x^{\ell_1+k} \ot y^{r-k} \right) \ot x^{\ell_2} \right] \\
&=(m_A \ot \mathbf{1}) \left[ \sum_{k=0}^r {r \choose k} \left(\frac{1}{2}\right)^k [\ell_1]^{[k]} \ x^{\ell_1+k} \ot \left( \sum_{s=0}^{r-k} {{r-k} \choose s} \left(\frac{1}{2}\right)^s [\ell_2+1]^{[s]} \ x^{\ell_2+s} \ot y^{r-k-s} \right) \right] \\
&=\sum_{k=0}^r \sum_{s=0}^{r-k} {r \choose k} {{r-k} \choose s} \left(\frac{1}{2}\right)^{k+s} [\ell_1]^{[k]} \  [\ell_2+1]^{[s]} \ x^{(\ell_1+\ell_2)+(k+s)} \ot y^{r-(k+s)} \\
&=\sum_{t=0}^r \left(\sum_{k+s=t} {r \choose k} {{r-k} \choose s} \left(\frac{1}{2}\right)^t [\ell_1]^{[k]} \  [\ell_2+1]^{[s]} \right) x^{(\ell_1+\ell_2)+t} \ot y^{r-t}.
\end{align*}

To show the left hand side of \eqref{compatible twisting2} is equal to its right hand side, it suffices to show 
$${r \choose t}[\ell_1+\ell_2+1]^{[t]} = \sum_{k=0}^t {r \choose k} {{r-k} \choose {t-k}} [\ell_1]^{[k]} \  [\ell_2+1]^{[t-k]}.$$
Using binomial identity in \cite[p.~4038]{CLW}, $[\ell_1+\ell_2+1]^{[t]} = \sum_{k=0}^t {t \choose k} [\ell_1]^{[k]} \ [\ell_2+1]^{[t-k]}$, we have 
\begin{align*}
{r \choose t}[\ell_1+\ell_2+1]^{[t]} &= \sum_{k=0}^t {r \choose t} {t \choose k} [\ell_1]^{[k]} \  [\ell_2+1]^{[t-k]} \\
&= \sum_{k=0}^t {r \choose k} {{r-k} \choose {t-k}} [\ell_1]^{[k]} \  [\ell_2+1]^{[t-k]},
\end{align*}
where the last equality is due to 
\begin{align*}
{r \choose t} {t \choose k} &= \frac{r!}{(r-t)!\, t!} \cdot \frac{t!}{(t-k)! \,k!} = \frac{r!}{(r-t)!\, (t-k)!\, k!} \\
&= \frac{r!}{k!\, (r-k)!} \cdot \frac{(r-k)!}{(t-k)! \, (r-t)!} = {r \choose k} {{r-k} \choose {t-k}}.
\end{align*}
This gives us the equality \eqref{compatible twisting2} as desired.\medskip

\noindent
(c) \textbf{Commutativity of diagram:} We need to check that the following diagrams commute: 
\[
\begin{xy}*!C\xybox{
\xymatrixcolsep{3.3pc}\xymatrix{
  B\ot A\ar[d]^{\tau_{\text{odd}}}\ar[r]^{\mathbf{1} \ot x \cdot} & B\ot A \ar[d]^{\tau_{\text{even}}} \\
 A\ot B\ar[r]^{x \cdot \ot \mathbf{1}} & A\ot B
}}
\end{xy} \quad  \quad \text{ and } \quad \quad
\begin{xy}*!C\xybox{
\xymatrixcolsep{3.3pc}\xymatrix{
  B\ot A\ar[d]^{\tau_{\text{even}}}\ar[r]^{\mathbf{1} \ot x^{p-1} \cdot } & B\ot A \ar[d]^{\tau_{\text{odd}}} \\
 A\ot B\ar[r]^{x^{p-1} \cdot \ot \mathbf{1}} & A\ot B.
}}
\end{xy}
\]

For the first diagram, we have:
\begin{align*}
\tau_{\text{even}} \circ (\mathbf{1} \ot x \cdot)(y^r \ot x^\ell) &= \tau (y^r \ot x^{\ell+1}) \\ 
&= \sum_{t=0}^r {r \choose t} \left(\frac{1}{2}\right)^t [\ell+1]^{[t]} \ x^{\ell+1+t} \ot y^{r-t}, \\
(x \cdot \ot \mathbf{1}) \circ \tau_{\text{odd}}(y^r \ot x^\ell) &= (x \cdot \ot \mathbf{1}) \left( \sum_{t=0}^r {r \choose t} \left(\frac{1}{2}\right)^t [\ell+1]^{[t]} \ x^{\ell+t} \ot y^{r-t} \right) \\
&= \sum_{t=0}^r {r \choose t} \left(\frac{1}{2}\right)^t [\ell+1]^{[t]} \ x^{\ell+1+t} \ot y^{r-t}.
\end{align*}

Similarly, for the second diagram:
\begin{align*}
\tau_{\text{odd}} \circ (\mathbf{1} \ot x^{p-1} \cdot)(y^r \ot x^\ell) &= \tau_{\text{odd}} (y^r \ot x^{\ell+p-1}) \\
&= \begin{cases}
   0 & \text{ if } \ell > 0 \\
   \tau_{\text{odd}} (y^r \ot x^{p-1})  & \text{ if } \ell = 0
\end{cases}  \\
&= \begin{cases}
   0 & \text{ if } \ell > 0 \\
   \sum_{t=0}^r {r \choose t} \left(\frac{1}{2}\right)^t [p]^{[t]} \ x^{p-1+t} \ot y^{r-t}  & \text{ if } \ell = 0
\end{cases} \\ 
&= \begin{cases}
   0 & \text{ if } \ell > 0 \\
   x^{p-1} \ot y^{r}  & \text{ if } \ell = 0
\end{cases} \\ 
(x^{p-1} \cdot \ot \mathbf{1}) \circ \tau_{\text{even}}(y^r \ot x^\ell) &= (x^{p-1} \cdot \ot \mathbf{1}) \left( \sum_{t=0}^r {r \choose t} \left(\frac{1}{2}\right)^t [\ell]^{[t]} \ x^{\ell+t} \ot y^{r-t} \right) \\
&= \sum_{t=0}^r {r \choose t} \left(\frac{1}{2}\right)^t [\ell]^{[t]} \ x^{\ell+t+p-1} \ot y^{r-t} \\
&= \begin{cases}
    0 & \text{ if } \ell > 0 \\
    x^{p-1} \ot y^{r} & \text{ if } \ell=0. \\
\end{cases}
\end{align*}
\end{proof}

For each $i$, the map $\tau_i$ is used to give $P_i(A) \ot P_j(B) = A\ot B$ the structure
of a left $(A\ot_{\tau} B)$-module.
In the case when $i$ is even, this is the usual $(A\ot _{\tau} B)$-module structure.
In the case when $i$ is odd, the $(A\ot_{\tau}B)$-module structure via $\tau_i$
is given by
\[
   \begin{xy}*!C\xybox{\xymatrixcolsep{2pc}
\xymatrix{
  & (A\ot_{\tau} B) \ot (A\ot B) \ar[rr]^{\quad \mathbf{1} \ot \tau_i\ot \mathbf{1}} 
   && A\ot A \ot B\ot B \ar[rr]^{\qquad m_A\ot m_B} && A\ot B.
}}
\end{xy}
\]

\begin{lemma}
Retaining the above module structure, $P_i(A) \ot P_j(B) = A\ot B$ is a free $(A\ot_{\tau}B)$-module 
of rank one, generated by $1\ot 1$, via the $(A\ot_{\tau}B)$-module isomorphism $\varphi: A \ot_\tau B \rightarrow A \ot B$ given by 
$$ \varphi(x^\ell \ot y^r) = \begin{cases} x^\ell \ot y^r, & \quad i \text{ is even} \\
\sum_{t=0}^r {r \choose t} \frac{t!}{2^t} \ x^{\ell+t} \ot y^{r-t}, & \quad i \text{ is odd,}
\end{cases}$$
whose inverse is given by
$$ \varphi^{-1}(x^\ell \ot y^r) = \begin{cases} x^\ell \ot y^r, & \quad i \text{ is even} \\
x^\ell \ot y^r - \frac{r}{2}x^{\ell+1} \ot y^{r-1}, & \quad i \text{ is odd.}
\end{cases}$$
\end{lemma}

\begin{proof}
When $i$ is even, $\varphi$ is clearly bijective.
We check that $\varphi$ is a bijection with the given inverse
when $i$ is odd:
\begin{align*}
\varphi \circ \varphi^{-1}(x^\ell \ot y^r) &= \varphi \left( x^\ell \ot y^r - \frac{r}{2}x^{\ell+1} \ot y^{r-1} \right) \\
&= \sum_{t=0}^r {r \choose t} \frac{t!}{2^t} \ x^{\ell+t} \ot y^{r-t} - \frac{r}{2} \sum_{k=0}^{r-1} {{r-1} \choose k} \frac{k!}{2^k} \ x^{\ell+1+k} \ot y^{r-1-k} \\
&= x^\ell \ot y^r + \sum_{t=1}^r {r \choose t} \frac{t!}{2^t} \ x^{\ell+t} \ot y^{r-t} - \sum_{t=1}^r r{{r-1} \choose {t-1}} \frac{(t-1)!}{2^t} \ x^{\ell+t} \ot y^{r-t} \\
&= x^\ell \ot y^r + \sum_{t=1}^r \left( {r \choose t}t - r{{r-1} \choose {t-1}} \right) \frac{(t-1)!}{2^t} \ x^{\ell+t} \ot y^{r-t} \\
&= x^\ell \ot y^r + \sum_{t=1}^r (0) \frac{(t-1)!}{2^t} \ x^{\ell+t} \ot y^{r-t} =x^\ell \ot y^r. 
\end{align*}
By \eqref{compatible twisting2}, $\varphi$ is a module isomorphism. Therefore $A\ot B$ is free as an $(A \ot_{\tau}B)$-module.
\end{proof}

In particular, the following is useful for our computations later:
\begin{equation}  \label{varphi}
\begin{split}
\varphi^{-1}(1 \ot y) &= \begin{cases} 1 \ot y, & \quad i \text{ is even} \\
1 \ot y - \frac{1}{2} x \ot 1, & \quad i \text{ is odd},
\end{cases} \\
\varphi^{-1}(1 \ot y^{p-1}) &= \begin{cases} 1 \ot y^{p-1}, & \quad i \text{ is even} \\ 
1 \ot y^{p-1} + \frac{1}{2} x \ot y^{p-2}, & \quad i \text{ is odd.} 
\end{cases}
\end{split}
\end{equation}

By Theorem~\ref{thm:twisted resoln}, the total complex $K_{\bu} := \mbox{Tot}(P_{\bu}(A)
\ot P_{\bu}(B))$ is a free resolution of $\k$ as $A\ot_{\tau}B$-module. 
For each $i,j \geq 0$, let $\phi_{i,j}$ denote the free generator $1\ot 1$
of $P_i(A)\ot P_j(B)$ as an $A\ot_{\tau}B$-module. 
Then as an $R$-module:
\[
    K_n = \bigoplus_{i+j=n} R \, \phi_{i,j} .
\]

Recall that the differentials of this total complex are 
$d_n = \sum_{i+j=n} (d_i\ot 1 + (-1)^i \ot d_j)$. 
As $P_i(A)\ot P_j(B)$ is free as an $A\ot_{\tau}B$-module with generator 
$\phi_{i,j}$, we can write the image of $\phi_{i,j}$ under the differential map 
as the action of $A\ot_{\tau}B$ on $\phi_{i,j}$, using values of the inverse map 
$\varphi^{-1}$ defined in (\ref{varphi}) where needed.
We express the differential on elements via this notation: 
\[
    d(\phi_{i,j}) = \left\{\begin{array}{ll}
    x^{p-1} \phi_{i-1,j} + y^{p-1} \phi_{i,j-1} & \mbox{ if $i,j$ are even}, \vspace{0.5em} \\
   x^{p-1}\phi_{i-1,j} + y \phi_{i,j-1} & \mbox{ if  $i$ is even and $j$ is odd}, \vspace{0.5em} \\
  x\phi_{i-1,j} - (y^{p-1}+ \frac{1}{2}xy^{p-2}) \phi_{i,j-1} & \mbox{ if $i$ is odd and $j$ is even}, \vspace{0.5em} \\
   x\phi_{i-1,j} - (y-\frac{1}{2}x)\phi_{i,j-1} & \mbox{ if $i,j$ are odd} . 
\end{array}\right.
\]
We interpret $\phi_{i,j}$ to be 0 if either $i$ or $j$ is negative.


\subsection{Resolution over the bosonization $R\# \kG$}
\label{subsec:resoln bosonization}

Again we take $\k$ to be a field of characteristic $p>2$ and
$R$, $G$ as described in Section~\ref{subsec:settingR}.
For the group $G=\langle g \rangle \cong \mathbb{Z}/q\mathbb{Z}$, where $q$ is 
divisible by $p$, we define an action of $G$ 
on the $R$-complex $K_{\bu}$ 
constructed in Section~\ref{subsec:resoln R}, 
for the purpose of forming a twisted tensor product resolution of $K_{\bu}$ 
with a resolution of $\k$ as $\kG$-module.
The group action will give us a twisting map on the complex 
$K_{\bu}$ analogous to the twisting map defining a skew group algebra. 
The resulting complex will give us a resolution of $\k$ over 
the bosonization $R\# \kG$. \\

Let $R=A\ot_{\tau} B$ be the twisted tensor product with $A:=\k[x]/(x^p)$, $B:=\k[y]/(y^p)$ and $\tau: B\ot A \rightarrow A\ot B$ as defined before. Under our setting in characteristic $p>2$ and relation $yx=xy+ \frac{1}{2}x^2$ in $R$, we obtain the following relations in $R$ (here, we drop the tensor symbols and write $xy$ in place of $x\ot y$ in $R$): 

\begin{prop} \label{propR}
The following relations hold in $R$: 
\begin{enumerate}
 \item For any integer $\ell \geq 0$, $\displaystyle yx^\ell = x^\ell y + \frac{\ell}{2}x^{\ell + 1}$. 
 \item For any integer $n \geq 1$, $\displaystyle (x+y)^n = \sum_{i=0}^n {n \choose i} \frac{(i+1)!}{2^i}\ x^i y^{n-i}$.
\end{enumerate}
\end{prop}

\begin{proof} 
(1) This was shown right after \cite[Lemma 3.8]{CLW}; we provide the proof here for completeness. We proceed by induction on $\ell$. Case $\ell=0$ is trivial. Case $\ell=1$ is from the relation $yx=xy+ \frac{1}{2}x^2$. Now assume the statement holds up to $\ell-1$. Then by the induction hypothesis, we have:
\begin{align*} 
yx^\ell &= (yx^{\ell-1})x = (x^{\ell-1} y + \frac{\ell-1}{2}x^\ell) x = x^{\ell-1} yx + \frac{\ell-1}{2}x^{\ell + 1} \\ 
&= x^{\ell-1}(xy+ \frac{1}{2}x^2) + \frac{\ell-1}{2}x^{\ell + 1} = x^\ell y + \frac{1}{2}x^{\ell + 1} + \frac{\ell-1}{2}x^{\ell + 1} = x^\ell y + \frac{\ell}{2}x^{\ell + 1}.
\end{align*}

(2) We proceed by induction on $n$. Case $n=1$ is trivial. Assume the assertion is true for up to $n-1$. Then by part (1), we have:
\begin{align*}
(x+y)^n &= (x+y) (x+y)^{n-1} = (x+y) \left( \sum_{i=0}^{n-1} {{n-1} \choose i} \frac{(i+1)!}{2^i}\ x^i y^{n-1-i} \right) \\
&= \sum_{i=0}^{n-1} {{n-1} \choose i} \frac{(i+1)!}{2^i}\ x^{i+1} y^{n-1-i} + \sum_{i=0}^{n-1} {{n-1} \choose i} \frac{(i+1)!}{2^i}\ (yx^i) y^{n-1-i} \\
&= \sum_{i=0}^{n-1} {{n-1} \choose i} \frac{(i+1)!}{2^i}\ x^{i+1} y^{n-1-i} + \sum_{i=0}^{n-1} {{n-1} \choose i} \frac{(i+1)!}{2^i}\ (x^iy+\frac{i}{2}x^{i+1}) y^{n-1-i} \\
&= \sum_{i=0}^{n-1} {{n-1} \choose i} \frac{(i+1)!}{2^i}\ x^{i+1} y^{n-1-i} + \sum_{i=0}^{n-1} {{n-1} \choose i} \frac{(i+1)!}{2^i}\ x^i y^{n-i} \\
 &\qquad \qquad \qquad \qquad \qquad + \sum_{i=0}^{n-1} {{n-1} \choose i} \frac{(i+1)!}{2^i}\ \frac{i}{2} \ x^{i+1}y^{n-1-i} \\
&= \sum_{i=0}^{n-1} {{n-1} \choose i} \frac{(i+2)!}{2^{i+1}} \ x^{i+1} y^{n-1-i} + \sum_{i=0}^{n-1} {{n-1} \choose i} \frac{(i+1)!}{2^i}\ x^i y^{n-i} \\
&= \sum_{i=1}^{n} {{n-1} \choose {i-1}} \frac{(i+1)!}{2^i} \ x^i y^{n-i} + \sum_{i=1}^{n-1} {{n-1} \choose i} \frac{(i+1)!}{2^i}\ x^i y^{n-i} + y^n \\
&= \sum_{i=0}^n {n \choose i} \frac{(i+1)!}{2^i}\ x^i y^{n-i},
\end{align*}
where the last equality is due to the recursive formula of the binomial coefficients 
$$\displaystyle {n \choose i}  = {{n-1} \choose {i-1}}+{{n-1} \choose i}, \text{ for all } 1 \leq i \leq n-1.$$ 
\end{proof}

\begin{lemma}
\label{alpha}
Set an element $$\alpha:= -y^{p-2} + \sum_{i=1}^{p-2} (-1)^{i+1} \frac{(i+1)!}{2^{i+1}} \ x^i y^{p-2-i} \in R.$$
Then $\alpha$ satisfies the following:
\begin{enumerate}
 \item[(a)] $x\alpha = (x+y)^{p-1} - y^{p-1} + \frac{1}{2}x[(x+y)^{p-2} - y^{p-2}]$ , \vskip0.3em
 \item[(b)] $(x+y)\alpha = -y^{p-1} - \frac{1}{2}xy^{p-2}$ , \vskip0.3em
 \item[(c)] $\alpha x = (x+y)^{p-1} - y^{p-1}$ , \vskip0.3em
 \item[(d)] $\alpha(y-\frac{1}{2}x) = - (x+y)^{p-1}$.
\end{enumerate}
\end{lemma}

\begin{proof}
\textbf{(a):} By Proposition~\ref{propR} (2), we have
\begin{align*}
&(x+y)^{p-1} - y^{p-1} + \frac{1}{2}x[(x+y)^{p-2} - y^{p-2}] \\
&= \sum_{i=1}^{p-1} {{p-1} \choose i} \frac{(i+1)!}{2^i}\ x^i y^{p-1-i} + \frac{1}{2}x \sum_{i=1}^{p-2} {{p-2} \choose i} \frac{(i+1)!}{2^i}\ x^i y^{p-2-i}\\
&= \sum_{i=0}^{p-2} {{p-1} \choose {i+1}} \frac{(i+2)!}{2^{i+1}}\ x^{i+1} y^{p-2-i} + \sum_{i=1}^{p-2} {{p-2} \choose i} \frac{(i+1)!}{2^{i+1}}\ x^{i+1} y^{p-2-i}\\
&= {{p-1} \choose 1}\ x y^{p-2} + \sum_{i=1}^{p-2} \left({{p-1} \choose {i+1}} \frac{(i+2)!}{2^{i+1}} + {{p-2} \choose i} \frac{(i+1)!}{2^{i+1}} \right)\ x^{i+1} y^{p-2-i}\\
&=x \left[- y^{p-2} + \sum_{i=1}^{p-2} \left(\frac{(p-1)!}{(i+1)!(p-i-2)!} \frac{(i+2)!}{2^{i+1}} + \frac{(p-2)!}{i! (p-i-2)!} \frac{(i+1)!}{2^{i+1}} \right)\ x^i y^{p-2-i} \right] \\
&=x \left[- y^{p-2} + \sum_{i=1}^{p-2} [(i+2)(p-i-1) \cdots (p-1) + (i+1)(p-i-1) \cdots (p-2) ] \frac{1}{2^{i+1}}\ x^i y^{p-2-i} \right] \\
&=x \left[- y^{p-2} + \sum_{i=1}^{p-2} [(-1)^{i+1} 1 \cdot 2 \cdots (i+2) + (-1)^i 1\cdot 2 \cdots (i+1)(i+1) ] \frac{1}{2^{i+1}}\ x^i y^{p-2-i} \right] \\
&=x \left[- y^{p-2} + \sum_{i=1}^{p-2} \frac{(-1)^{i+1}(i+1)!}{2^{i+1}}(i+2 -i-1) \ x^i y^{p-2-i} \right] = x \alpha. \\
\end{align*}

\textbf{(b):} By Proposition~\ref{propR} (2), we have
\begin{align*}
(x+y)\alpha &= (x+y) \left( - y^{p-2} + \sum_{i=1}^{p-2} \frac{(-1)^{i+1}(i+1)!}{2^{i+1}} \ x^i y^{p-2-i}  \right) \\
&= - xy^{p-2} + \sum_{i=1}^{p-2} \frac{(-1)^{i+1}(i+1)!}{2^{i+1}} \ x^{i+1} y^{p-2-i} - y^{p-1} \\
    &\qquad \qquad \ + \sum_{i=1}^{p-2} \frac{(-1)^{i+1}(i+1)!}{2^{i+1}} \ (x^iy+\frac{i}{2}x^{i+1}) y^{p-2-i} \\
&= - xy^{p-2} + \sum_{i=2}^{p-1} \frac{(-1)^{i}\, i!}{2^i} \ x^i y^{p-1-i} - y^{p-1} + \sum_{i=1}^{p-2} \frac{(-1)^{i+1}(i+1)!}{2^{i+1}} \ x^iy^{p-1-i} \\
    &\qquad \qquad \ + \sum_{i=2}^{p-1} \frac{(-1)^{i} \, i!}{2^i} \frac{i-1}{2} x^i y^{p-1-i} \\
&= - xy^{p-2} - y^{p-1} + \sum_{i=1}^{p-2} \frac{(-1)^{i+1}(i+1)!}{2^{i+1}} \ x^iy^{p-1-i} + \sum_{i=2}^{p-1} \frac{(-1)^{i} \, i!}{2^i} \frac{i+1}{2} \ x^i y^{p-1-i} \\
&= - xy^{p-2} - y^{p-1} + \frac{2!}{2^2}xy^{p-2} + \sum_{i=2}^{p-2} \left( \frac{(-1)^{i+1}(i+1)!}{2^{i+1}} +\frac{(-1)^{i}(i+1)!}{2^{i+1}} \right) x^i y^{p-1-i} \\
    &\qquad \qquad \qquad + \frac{(-1)^{p-1} \, p!}{2^p}  x^p y^{p-1-(p-1)} \\
&= - y^{p-1} - \frac{1}{2}xy^{p-2}. \\
\end{align*}

We are going to check for identities (c) and (d) in the domain $\k \langle x,y \rangle / (yx-xy-\frac{1}{2}x^2)$. \\

\textbf{(c):} Since $(x+y)$ is a regular element in this domain, to check (c) it suffices to show $(x+y)\alpha x =(x+y)[(x+y)^{p-1} - y^{p-1}]$. As identity (b) holds in the domain and by the relation (3.9) in \cite[Lemma 3.8]{CLW}, we have:
\begin{align*}
 (x+y)\alpha x &= (- y^{p-1} - \frac{1}{2}xy^{p-2}) x = - y^{p-1}x - \frac{1}{2}xy^{p-2}x \\
&= - \sum_{t=0}^{p-1} {{p-1} \choose t } \left(\frac{1}{2} \right)^t [1]^{[t]} \, x^{1+t}y^{p-1-t} - \frac{1}{2}x \sum_{s=0}^{p-2} {{p-2} \choose s } \left(\frac{1}{2} \right)^s [1]^{[s]} \, x^{1+s}y^{p-2-s} \\
&= - \sum_{t=0}^{p-1} {{p-1} \choose t } \left(\frac{1}{2} \right)^t t! \, x^{1+t}y^{p-1-t} - \sum_{s=0}^{p-2} {{p-2} \choose s } \left(\frac{1}{2} \right)^{s+1} s! \, x^{2+s}y^{p-2-s} \\
&= - \sum_{i=1}^{p} {{p-1} \choose {i-1}} \left(\frac{1}{2} \right)^{i-1} (i-1)! \, x^iy^{p-i} - \sum_{i=2}^{p} {{p-2} \choose {i-2}} \left(\frac{1}{2} \right)^{i-1} (i-2)! \, x^iy^{p-i} \\
&= -xy^{p-1} - \sum_{i=2}^{p} \left(\frac{(p-1)!}{(i-1)! (p-i)!} (i-1)! +\frac{(p-2)!}{(i-2)! (p-i)!} (i-2)! \right) \left(\frac{1}{2} \right)^{i-1} x^iy^{p-i} \\
&= -xy^{p-1},
\end{align*}
and 
\begin{align*}
(x+y)[(x+y)^{p-1} - y^{p-1}] &= (x+y)^p - xy^{p-1} - y^p =  \sum_{i=0}^p {p \choose i} \frac{(i+1)!}{2^i}\ x^i y^{p-i} - xy^{p-1} - y^p \\
&=y^p - xy^{p-1} - y^p = - xy^{p-1}.
\end{align*}
Thus, as $(x+y)\alpha x =(x+y)[(x+y)^{p-1} - y^{p-1}]$, we have identity (c) as claimed. \\

\textbf{(d):} By a similar technique as in (c), to check (d) it suffices to show 
$$(x+y)\alpha(y-\frac{1}{2}x) =(x+y)[ - (x+y)^{p-1}].$$ 
As we showed identities (b) and $(x+y)\alpha x=-xy^{p-1}$ hold above, we have
\begin{align*}
(x+y)\alpha(y-\frac{1}{2}x) &= \left(-y^{p-1} - \frac{1}{2}xy^{p-2} \right)y -\frac{1}{2} (x+y)\alpha x \\
&= -y^p - \frac{1}{2}xy^{p-1} - \frac{1}{2}(-xy^{p-1}) = -y^p = -(x+y)^p = (x+y)[ - (x+y)^{p-1}].
\end{align*}
Thus, identity (d) holds. 
\end{proof}

\vspace{1em}

Recall that $G=\langle g \rangle \cong \mathbb{Z}/q\mathbb{Z}$ acts on $R$ by ${}^g x = x$ and ${}^g y = x+y$. We define an action of $G$ on the complex $K_{\bu}$, constructed in Section~\ref{subsec:resoln R}, as follows:
$$
    {}^g \phi_{i,j} = \left\{ \begin{array}{ll}
       \phi_{i,j} , & \mbox{ if } i \mbox{ is odd} \\
       \phi_{i,j} + \phi_{i+1,j-1} , & \mbox{ if } i \mbox{ is even and } j \mbox{ is odd}\\
       \phi_{i,j} + \alpha \, \phi_{i+1,j-1} , & \mbox{ if } i,j \mbox{ are even},
    \end{array}\right. 
$$
and in general for all $0 \leq s \leq q-1$
\begin{align*}
    {}^{g^s} \phi_{i,j} = \left\{ \begin{array}{ll}
    \phi_{i,j} , & \mbox{ if } i \mbox{ is odd} \\
    \phi_{i,j} + s\, \phi_{i+1,j-1} , & \mbox{ if } i \mbox{ is even and } j \mbox{ is odd}\\
    \phi_{i,j} + ({}^{(g^{s-1}+ \cdots + g + 1)} \alpha) \, \phi_{i+1,j-1} , & \mbox{ if } i,j \mbox{ are even} .
    \end{array}\right. 
\end{align*}
Moreover, it is not hard to figure out the following for all $1 \leq s \leq q$
\begin{align}\label{G-Phi}
    {}^{g^{-s}} \phi_{i,j} = \left\{ \begin{array}{ll}
    \phi_{i,j} , & \mbox{ if } i \mbox{ is odd} \\
    \phi_{i,j} -s\, \phi_{i+1,j-1} , & \mbox{ if } i \mbox{ is even and } j \mbox{ is odd}\\
    \phi_{i,j} - ({}^{(g^{-s}+ \cdots + g^{-1})} \alpha) \, \phi_{i+1,j-1} , & \mbox{ if } i,j \mbox{ are even} .
    \end{array}\right. 
\end{align}

\begin{lemma}
With the above $G$-action, the complex $K_{\bu}$ is $G$-equivariant. 
\end{lemma}

\begin{proof}
We first show that this $G$-action is well-defined, that is, ${}^{g^q} \phi_{i,j} = \phi_{i,j}$. 

It is clear that when $i$ is odd, or when $i$ is even and $j$ is odd that ${}^{g^q} \phi_{i,j} = \phi_{i,j}$, as $q$ is divisible by the characteristic $p$ of the field $\k$. When $i,j$ are both even, we need to show ${}^{g^q} \phi_{i,j} = \phi_{i,j} + ({}^{(g^{q-1}+ \cdots + g + 1)} \alpha) \, \phi_{i+1,j-1} = \phi_{i,j}$, that is, we need to show ${}^{(g^{q-1}+ \cdots + g + 1)} \alpha = 0$. 

From Lemma~\ref{alpha} (c), we have $\alpha x = (x+y)^{p-1} - y^{p-1}$. Now apply $g^s$-action on both sides, we have:
\begin{align*}
{}^{g^s} \alpha \ {}^{g^s} x &= ({}^{g^s}x + {}^{g^s}y)^{p-1} - ({}^{g^s}y)^{p-1} \\
{}^{g^s} \alpha \ x &= (x+y+sx)^{p-1} - (y+sx)^{p-1}. 
\end{align*}
Thus, summing over all $0 \leq s \leq q-1$: 
\begin{align*}
\left(\sum_{s=0}^{q-1} {}^{g^s}\alpha \right) x &= \sum_{s=0}^{q-1} \left([y+(s+1)x]^{p-1} - (y+sx)^{p-1}\right) \\
&= (y+qx)^{p-1}-y^{p-1} = 0.
\end{align*}
So $\sum_{s=0}^{q-1} {}^{g^s}\alpha=0$ in the domain $\k \langle x,y \rangle / (yx-xy-\frac{1}{2}x^2)$ and hence is also $0$ in $R$. Therefore, in all cases, we have ${}^{g^q} \phi_{i,j} = \phi_{i,j}$ and the above $G$-action is well-defined. \\

To check that complex $K_{\bu}$ is $G$-equivariant, we need to check such $G$-action is compatible with the differential maps in each degree, that is, $d(\,^g \phi_{i,j}) = \,^gd(\phi_{i,j})$, for all $i,j \geq 0$.\medskip 

\noindent
\textbf{When $i,j$ are even:} 
\begin{align*}
d(\,^g \phi_{i,j}) &= d(\phi_{i,j} + \alpha \phi_{i+1,j-1}) \\
&= x^{p-1}\phi_{i-1,j} + y^{p-1}\phi_{i,j-1} + \alpha \left(x \phi_{i,j-1}-(y-\frac{1}{2}x)\phi_{i+1,j-2} \right), \\
^gd(\phi_{i,j}) &= \,^g x^{p-1} \ ^g\phi_{i-1,j} + \,^gy^{p-1} \ ^g\phi_{i,j-1} \\
&= x^{p-1} \phi_{i-1,j} + (x+y)^{p-1}(\phi_{i,j-1}+\phi_{i+1,j-2}).
\end{align*}
By comparing the coefficients, we see that coefficients for the terms $\phi_{i,j-1}$ and $\phi_{i+1,j-2}$ are exactly identities (c) and (d) in Lemma~\ref{alpha}, respectively.\medskip

\noindent
\textbf{When $i$ is even and $j$ is odd:} 
\begin{align*}
d(\,^g \phi_{i,j}) &= d(\phi_{i,j} + \phi_{i+1,j-1}) \\
&= x^{p-1}\phi_{i-1,j} + y\phi_{i,j-1} + x\phi_{i,j-1} - (y^{p-1}+\frac{1}{2}xy^{p-2})\phi_{i+1,j-2}, \\
^gd(\phi_{i,j}) &= \,^gx^{p-1} \ ^g\phi_{i-1,j} + \,^gy \ ^g\phi_{i,j-1} \\
&= x^{p-1}\phi_{i-1,j} + (x+y)(\phi_{i,j-1}+\alpha\phi_{i+1,j-2}).
\end{align*}
The coefficient for the term $\phi_{i+1,j-2}$ is exactly the identity (b) in Lemma~\ref{alpha}.\medskip

\noindent
\textbf{When $i$ is odd and $j$ is even:} 
\begin{align*}
d(\,^g \phi_{i,j}) &= d(\phi_{i,j}) = x\phi_{i-1,j} - (y^{p-1}+\frac{1}{2}xy^{p-2})\phi_{i,j-1}, \\
^gd(\phi_{i,j}) &= \,^gx \ ^g\phi_{i-1,j} - (\,^gy^{p-1}+\frac{1}{2} \,^gx \ ^gy^{p-2}) \ ^g\phi_{i,j-1} \\
&= x[\phi_{i-1,j} + \alpha\phi_{i,j-1}] - [(x+y)^{p-1} + \frac{1}{2}x(x+y)^{p-2}] \phi_{i,j-1}.
\end{align*}
The coefficient for the term $\phi_{i,j-1}$ is exactly the identity (a) in Lemma~\ref{alpha}.\medskip

\noindent
\textbf{When $i,j$ are odd:} 
\begin{align*}
d(\,^g \phi_{i,j}) &= d(\phi_{i,j}) = x\phi_{i-1,j} - (y-\frac{1}{2}x)\phi_{i,j-1}, \\
^gd(\phi_{i,j}) &= \,^gx \ ^g\phi_{i-1,j} - (\,^gy - \frac{1}{2}\ ^gx)\ ^g\phi_{i,j-1} \\
&= x (\phi_{i-1,j} + \phi_{i,j-1}) - (x+y- \frac{1}{2}x) \phi_{i,j-1}.
\end{align*}
In all cases, we have $d(\,^g \phi_{i,j}) = \,^gd(\phi_{i,j})$. Thus, complex $K_{\bu}$ is $G$-equivariant. 
\end{proof}

We use this $G$-action next to form a twisted tensor product resolution of $\k$ as an $R \# \kG$-module. Let the twisting map $\tau'_n: \kG\ot K_n \rightarrow K_n\ot \kG$ be given by the action of $G$ on $K_n$, so that
$$\tau'_{i+j} (g\ot \phi_{i,j}) = \left\{\begin{array}{ll}
     \phi_{i,j}\ot g & \mbox{ if } i \mbox{ is odd}\\
    (\phi_{i,j} + \phi_{i+1,j-1}) \ot g &\mbox{ if } i \mbox{ is even and } j \mbox{ is odd}\\
    (\phi_{i,j} + \alpha \, \phi_{i+1,j-1})\ot g & \mbox{ if } i,j \mbox{ are even}.
   \end{array} \right.$$
Then $K_{\bu}$ is compatible with $\tau'$ (giving the twisting that governs
the smash product construction)  via the maps $\tau_n'$.

Let $P_{\bu}(\kG)$ be the following free resolution of the trivial $\kG$-module $\k$:
\[
\begin{xy}*!C\xybox{\xymatrixcolsep{1.5pc}
\xymatrix{
   P_{\bu}(\kG): \quad \cdots \ar[rr]^{\qquad \ \left( \sum_{s=0}^{q-1} g^s \right) \cdot} && \kG\ar[rr]^{(g-1)\cdot} &&
    \kG\ar[rr]^{ \left( \sum_{s=0}^{q-1} g^s  \right) \cdot} && \kG\ar[rr]^{(g-1)\cdot} && \kG\ar[r]^{\varepsilon} & \k\ar[r] & 0 , 
}}
\end{xy}
\]
where $\varepsilon$ takes $g$ to 1.

Let $Y_{\bu} := \Tot (K_{\bu}\ot P_{\bu}(\kG) )$.
By \cite[Lemma 5.9]{SW}, $Y_{\bu}\rightarrow \k$ is exact.
The modules are free $(R\# \kG)$-modules by a similar argument to what we used earlier:
In each degree we have a direct sum of modules of the form $R\ot \kG$.
Each such is freely generated by some $\phi_{i,j}\ot \phi_k$, 
where $\phi_k$ denotes the free generator for $P_k(\kG) = \kG$. 
So by Theorem~\ref{thm:twisted resoln}, 
$Y_{\bu}$ is a free resolution of the $(R\# \kG)$-module $\k$.

For each $i,j,k \geq 0$, let $\phi_{i,j,k}$ denote the free generator 
$\phi_{i,j} \ot \phi_k$ of $K_{i+j}\ot P_k(\kG)$ 
as an $(R\# \kG)$-module. 
We set $\phi_{i,j,k}=0$ if one of $i,j,k$ is negative. 
Then, for all $n \geq 0$, as an $(R\# \kG)$-module,
\[
   Y_n = \bigoplus_{i+j+k=n} (R\# \kG) \, \phi_{i,j,k} .
\]
We express the differentials on elements via this notation. 
\[
\begin{aligned}
 & d(\phi_{i,j,k}) = d(\phi_{i,j})\ot \phi_k \\
 & + (-1)^{i+j} \left\{ \begin{array}{ll}
   (g-1)\,\phi_{i,j,k-1} , & \mbox{ if } i,k \mbox{ are odd} \vspace{0.7em} \\
   (g-1)\,\phi_{i,j,k-1} - g\, \phi_{i+1,j-1,k-1} , &\mbox{ if } i \mbox{ is even and } j,k \mbox{ are odd} \vspace{0.7em} \\
   (g-1)\,\phi_{i,j,k-1} - \alpha g\,\phi_{i+1,j-1,k-1}, &\mbox{ if } i,j \mbox{ are even and } k \mbox{ is odd} \vspace{0.7em} \\
   \left( \sum_{s=0}^{q-1} g^s \right) \phi_{i,j,k-1} , & \mbox{ if } i \mbox{ is odd and } k \mbox{ is even} \vspace{0.8em} \\
   \left( \sum_{s=0}^{q-1} g^s \right) \phi_{i,j,k-1} -\left( \sum_{s=0}^{q-1} sg^s \right)\,\phi_{i+1,j-1,k-1} , &\mbox{ if } i,k \mbox{ are even and } j \mbox{ is odd} \vspace{0.7em} \\
   \left( \sum_{s=0}^{q-1} g^s \right) \phi_{i,j,k-1} \\ - \sum_{s=1}^{q-1} ({}^{(g^{s-1} + \cdots + g +1)} \alpha) \, g^s \phi_{i+1,j-1,k-1}, & \mbox{ if } i,j,k \mbox{ are even}.
   \end{array}\right.
\end{aligned}
\]
We will give partial verification to the above differentials in view of \eqref{G-Phi}.\medskip

\noindent
{\bf The case for $i,k$ are even and $j$ is odd.}
\begin{align*}
d(\phi_{i,j,k})&\, =d(\phi_{i,j})\otimes \phi_k+(-1)^{i+j}\phi_{i,j}\otimes d(\phi_k)=d(\phi_{i,j})\otimes \phi_k+(-1)^{i+j}\phi_{i,j}\left( \sum_{s=0}^{q-1} g^s \right)\otimes \phi_{k-1}\\
&\,=d(\phi_{i,j})\otimes \phi_k+(-1)^{i+j}\sum_{s=0}^{q-1}g^s\left(\!^{g^{-s}}\phi_{i,j}\right)\otimes \phi_{k-1}\\
&\,=d(\phi_{i,j})\otimes \phi_k+(-1)^{i+j}\sum_{s=0}^{q-1}g^s\left(\phi_{i,j}-s\phi_{i+1,j-1}\right)\otimes \phi_{k-1}\\
&\,=d(\phi_{i,j})\otimes \phi_k+(-1)^{i+j}\left[\left( \sum_{s=0}^{q-1} g^s \right)\phi_{i,j,k-1}-\left(\sum_{s=0}^{q-1}sg^s\right)\phi_{i+1,j-1,k-1}\right]\\
\end{align*}

\noindent
{\bf The case for $i,j,k$ are even.}
\begin{align*}
d(\phi_{i,j,k})&\, =d(\phi_{i,j})\otimes \phi_k+(-1)^{i+j}\phi_{i,j}\otimes d(\phi_k)=d(\phi_{i,j})\otimes \phi_k+(-1)^{i+j}\phi_{i,j}\left(  \sum_{s=0}^{q-1} g^s \right)\otimes \phi_{k-1}\\
&\,=d(\phi_{i,j})\otimes \phi_k+(-1)^{i+j}\sum_{s=0}^{q-1}g^s\left(\!^{g^{-s}}\phi_{i,j}\right)\otimes \phi_{k-1}\\
&\,=d(\phi_{i,j})\otimes \phi_k+(-1)^{i+j}\left[\left(\phi_{i,j}+\sum_{s=1}^{q-1}g^s(\phi_{i,j}-\!^{(g^{-1}+\cdots+g^{-s})}\alpha\, \phi_{i+1,j-1})\right)\otimes \phi_{k-1}\right]\\
&\,=d(\phi_{i,j})\otimes \phi_k+(-1)^{i+j}\left[\left( \sum_{s=0}^{q-1} g^s \right)\phi_{i,j,k-1}-\left(\sum_{s=1}^{q-1}g^s\,(\!^{(g^{-1}+\cdots+g^{-s})}\alpha)\right)\phi_{i+1,j-1,k-1}\right]\\
&\,=d(\phi_{i,j})\otimes \phi_k+(-1)^{i+j}\left[\left( \sum_{s=0}^{q-1} g^s \right)\phi_{i,j,k-1}-\left(\sum_{s=1}^{q-1}\!^{(g^{s-1}+\cdots+g+1)}\alpha\, g^s\right)\phi_{i+1,j-1,k-1}\right].
\end{align*}

Note that in the above formulas for the differentials, we have $\varepsilon\left(\sum_{s=0}^{q-1}g^s\right)=0=\varepsilon\left(\sum_{s=0}^{q-1}sg^s\right)$ in characteristic $p$, since $p$ divides $q$. Among these differentials of free $R\#\kG$-module basis elements, the only terms in the
outcomes $d(\phi_{i,j,k})$ that do not have coefficients in the
augmentation ideal $\Ker(\varepsilon)$ are the terms $ - g\phi_{i+1,j-1,k-1}$, occurring 
when $i$ is even and $j,k$ are odd. 

Consequently, letting $n=i+j+k$ and $\phi^*_{i,j,k}$ be the dual basis vector
to $\phi_{i,j,k}$ in 

\noindent
$\Hom_\k \left( \bigoplus_{i'+j'+k' = n} \k \phi_{i',j',k'} , \k \right)
\cong \Hom_{R\# \kG} (Y_n,\k)$, we have
\[
    d^*(\phi^*_{i,j,k}) = \left\{ \begin{array}{ll}
    - \phi^*_{i-1,j+1,k+1} , & \mbox{ if } i \mbox{ is odd and } j,k \mbox{ are even}\\
    0 , & \mbox{ otherwise}.
    \end{array}\right.
\]
The cocycles are thus all the $\phi^*_{i,j,k}$ except those for which $i$ is odd
and $j,k$ are even.
The coboundaries are the $\phi^*_{i,j,k}$ for which $i$ is even and $j,k$ are odd.
Therefore, for all $n \geq 0$, as a vector space, 
\[
    \coh^n(R\# \kG ,\k) \cong \left\{
    \begin{array}{ll}
     \Span_\k\{\phi^*_{i,j,k}\mid i+j+k=n\} \\ -  \Span_\k\{\phi^*_{i,j,k}\mid i\mbox{ is even and }
     j,k \mbox{ are odd}\}, & \mbox{ if } n \mbox{ is even} \vspace{0.8em} \\
        \Span_\k\{\phi^*_{i,j,k}\mid i+j+k=n\} \\ -  \Span_\k\{\phi^*_{i,j,k}\mid i \mbox{ is odd and }
    j,k \mbox{ are even}\}, & \mbox{ if } n \mbox{ is odd}.
     \end{array}\right.
\]


\section{Anick resolutions}
\label{sec:Anick}

In this section, we recall the Anick resolution
and make some additional observations about it
in our setting. 
We will use the Anick resolution in Section~\ref{subsec:cohomology lifting}.


\subsection{The resolution construction}
\label{subsec:Anick} 

We generally construct the Anick resolution~\cite{Anick} 
as envisioned by Cojocaru and Ufnarovski~\cite{CU},
adapted here to left modules under some conditions.
An algorithmic description using Gr\"obner bases
is given by Green and Solberg~\cite{Green-Solberg}.
The construction of the resolution also serves as a proof of exactness,
since the differentials are defined recursively in each degree, making use of 
a contracting homotopy in the previous degree that is constructed 
recursively as well. See Theorem~\ref{thm:Anick} below, 
due to Anick.
We include a proof in our setting because we will 
use the construction in Sections~\ref{subsec:tp} and~\ref{subsec:cohomology lifting}.

Let $A= T(V)/(I)$ where $V$ is a 
finite dimensional vector space, $T(V) = T_{\k}(V)$ is the tensor
algebra on $V$ over $\k$, and $I$ is a set of relations generating
an ideal $(I)$. 
We denote the image of an element $v$ of $V$ in $A$ also by $v$
when it will cause no confusion. 
We assume that $A$ is augmented by an algebra homomorphism 
$\varepsilon: A\rightarrow \k$ with $\varepsilon(v) =0$ for all $v\in V$.
Fix a totally ordered basis $v_1, \ldots, v_n$ of $V$ 
(say $v_1 < \cdots < v_n$) and
consider the degree lexicographic ordering on words in $v_1,\ldots, v_n$.
That is, we give each $v_i$ the degree~1, and 
monomials (words) are ordered first according to total degree, then monomials
having the same degree (i.e.\ word length) are ordered as in a dictionary.

A {\em normal word} (called an element of an {\em order ideal of monomials},
or {\em o.i.m.}\ in \cite{Anick}) is a monomial (considered as an element of $T(V)$)
that cannot be written as a linear combination of smaller words in $A$.
As a vector space, $A$ has a basis in one-to-one correspondence with
the set of normal words.

A {\em tip} (called an {\em obstruction} in \cite{Anick}) is a word 
(considered as an element in $T(V)$)
that is not normal but for which any proper subword is normal.
It follows that the tips correspond to the relations:
Let $u$ be a tip and write the image of $u$ in $A$ as a linear
combination $u = \sum a_i t_i$, where each $t_i$ is a normal
word and $a_i$ is a scalar. 
Then, viewed as an element of $T(V)$, $u - \sum a_i t_i$ is in
the ideal of relations, $(I)$. 
It also follows that the tips are in one-to-one correspondence
with a Gr\"obner basis of $(I)$~\cite{Green-Solberg}. 

We construct the Anick resolution from the chosen sets of generators
and relations in $A$ as follows. 
We reindex in comparison to \cite{Anick} so that indices
for spaces correspond
to homological degrees, and indices for functions correspond 
to homological degrees of their domains.

The Anick resolution is a free resolution of $\k$ considered to be an $A$-module
under the augmentation $\varepsilon$.
We will first describe a free basis $C_n$ in each homological degree $n$
of the resolution. 
We will write the resolution as: 
$$\begin{xy}*!C\xybox{
\xymatrix{
  \cdots \ar[r]^{d_3 \quad \ } & A\ot \k C_2
         \ar[r]^{ d_2 \ } & A\ot \k C_1
         \ar[r]^{ \quad \  d_1 } & A 
         \ar[r]^{\varepsilon  } & \k 
         \ar[r] & 0,
}}
\end{xy}$$
where $\k C_n$ denotes the vector space with basis $C_n$.
We adapt the degree lexicographic ordering on monomials in $T(V)$ to each
$A$-module $A\ot\k C_n$
by giving an element $s\ot t$,
where $s$ is a normal word and $t\in C_n$, the 
degree of $st$ viewed as an element of $T(V)$.

Let 
\[
  C_1 = \{ v_1,\ldots, v_n\} ,
\]
that is, $C_1$ is the chosen set of generators. 
Let $C_2$ be the set of tips (or obstructions).
The remaining sets $C_n$ will be defined as sets of paths of length $n$ in 
a directed graph (or quiver) associated to the generators and tips 
as follows~\cite{CU}.
The graph will have at most one directed arrow joining two vertices, and
paths will be denoted by the product of their vertices in $T(V)$,
written from right to left, for example, if
$f,g$ are vertices and there is an arrow from $f$ to $g$,
we denote the arrow by $gf$, and if there is a further
arrow from $g$ to $h$, then $hgf$ denotes the path
\[
   f\rightarrow g \rightarrow h
\]
starting at $f$, passing through $g$, and ending at $h$. 

Let ${\mathcal{B}} = \{ v_1,\ldots, v_n\}$ be a basis of $V$, 
equipped with the ordering $v_1<\cdots < v_n$ as above,
so that we may identify $\mathcal{B}$ with $C_1$.
Let $\mathcal{T}$ be the set of tips.  
Let $\mathcal{R}$ be the set of all proper prefixes (that is, left
factors) of the tips considered as elements of $T(V)$.
(Note that ${\mathcal{B}}\subset {\mathcal{R}}$.)
Let ${\mathbf{Q}} = {\mathbf{Q}}({\mathcal{B}},{\mathcal{T}})$
be the following quiver.
The vertex set is $ \{1\}\cup  {\mathcal{R}}$. 
The arrows are all $1\rightarrow v_i$ for $v_i\in {\mathcal {B}}$
and all $f\rightarrow g$ for which
the word $gf$ (viewed as an element of $T(V)$)
uniquely contains a tip, and that tip is a prefix
(possibly coinciding with $gf$).

The set $C_n$ consists of all paths of length $n$ starting from $1$
in the quiver.
In this context, the path 
$1\rightarrow f\rightarrow g$ is identified with the product $gf$.
(Note that $f\rightarrow g$ does not occur on its own
as an element of any $C_i$ if $f\neq 1$,
so for our purposes there will be no confusion in denoting paths this way.) 
For use in constructing the chains $C_n$, we observe that we only use 
the vertices that are in the connected component of $\mathbf{Q}$ 
containing $1$.
Let $\overline{\mathbf{Q}} =\overline{\mathbf{Q}}({\mathcal{B}},{\mathcal{T}})$ be the connected component of $1$
in $\mathbf{Q}$, called the {\em reduced quiver} of $\mathcal B$
and $\mathcal T$.

The differentials $d$ are defined recursively,
with a simultaneous recursive definition of a 
$\k$-linear contracting homotopy $s$:
\begin{equation}\label{eqn:Anick}
\begin{xy}*!C\xybox{
\xymatrix{
  \cdots \ar@<2pt>[r]^{d_3 \ \ \ } & A\ot \k C_2\ar@<2pt>[l]^{s_2 \ \ \ }
         \ar@<2pt>[r]^{ \ d_2  } & A\ot \k C_1\ar@<2pt>[l]^{ \ s_1  }
         \ar@<2pt>[r]^{\ \ \ \ d_1 } & A \ar@<2pt>[l]^{\ \ \ \ s_0  }
         \ar@<2pt>[r]^{\varepsilon  } & \k \ar@<2pt>[l]^{\eta  }
         \ar[r] & 0 ,
}}
\end{xy}
\end{equation}
where $\eta$ is the unit map (taking the multiplicative identity
of $\k$ to the multiplicative identity of $A$). 
We give these definitions next in our setting, simultaneously
proving the following theorem. 
Examples are given in \cite{Anick,CU} and below in 
Sections~\ref{subsec:tp} and~\ref{subsec:cohomology lifting}.

\begin{thm}\label{thm:Anick}\cite[Theorem~1.4]{Anick} 
There are maps $d_n, s_n$ for which $(A\ot \k C_{\bu} , d_{\bu})$ is a free
resolution of $\k$ as an $A$-module and $s_{\bu}$ is a contracting homotopy.
\end{thm}

\begin{proof}
We first define the maps $d_n, s_{n-1}$ for $n=1,2$ to illustrate
the general method.
We then use induction on $n$. \vskip 7pt

{\bf Degree 1:}
We take $n=1$ and let
\[
    d_1 (1\ot v_i) = v_i
\]
for all $v_i$ in $\mathcal{B}$ and extend $d_1$ so that it is a left $A$-module
homomorphism.
To define the $\k$-linear map $s_0: A\rightarrow A\ot \k C_1$, first 
write elements of $A$ as $\k$-linear combinations of normal words
(which form the chosen vector space basis of $A$).
Define $s_0$ on $A$ via its values on all normal words, which are as follows.
Set $s_0(1) =0$ and $s_0(uv_i)= u\ot v_i$ for all normal
words of the form $uv_i$ for some word $u$ and $v_i$ in $\mathcal{B}$.
Extend $s_0$ so that it is a  $\k$-linear map on $A$, 
and note that it will not be
an $A$-module homomorphism in general. 
We now see that by construction, 
\[
   a = (d_1 s_0 + \eta\varepsilon) (a)
\]
for all $a\in A$.
It also follows that $A\ot \k C_1 = \Ker (d_1)\oplus \ima (s_0)$.
To see this, let 
$b\in A\ot \k C_1$ and write $b= (b-s_0d_1 (b)) + s_0 d_1(b)$.
One checks that 
$b-s_0d_1(b)\in \Ker (d_1)$; by definition,  $s_0d_1(b)\in \ima (s_0)$.
The intersection of these two spaces is 0 by the above equation 
and definitions:
If $b\in \Ker (d_1)\cap \ima (s_0)$, write $b = s_0(c)$.
Then 
\[
   c = (d_1 s_0 + \eta \varepsilon) (c) = \eta\varepsilon(c) , 
\]
which implies $c\in\k$ so that $b = s_0(c) =0$. \vskip 7pt

{\bf Degree 2:}
We take $n=2$ and 
define $d_2( 1\ot u)$ for $u$ in $C_2$ as follows.
By definition of $C_2$, we may 
write $u = rv_i$ uniquely in $T(V)$ for a word $r$ 
in ${\mathcal{R}}$ and $v_i\in C_1$.
Consider $r\ot v_i$ as an element of $A\ot \k C_1$, and further take 
its image under the $A$-module homomorphism $d_1$: 
\[
   d_1 (r\ot v_i) = rd_1 (1\ot v_i) = rv_i .
\]
Define
\[
   d_2 (1\ot u) = r\ot v_i - s_0(d_1(r\ot v_i)) 
     = r\ot v_i - s_0(rv_i) ,
\]
and extend $d_2$ so that it is an $A$-module homomorphism on $A\ot \k C_2$. 
By its definition, $rv_i$  when considered as an element of $T(V)$
is a tip (not a normal word),
and considered here as an element of $A$, it
must be rewritten as a $\k$-linear combination of
normal words before applying $s_0$ (since $s_0$ is a $\k$-linear map
but not an $A$-module homomorphism).
Now the definitions of $d_1$ and $s_0$ imply
$d_1 s_0 |_{\Ker(\varepsilon)} = \mathbf{1}_{\Ker (\varepsilon)}$, the
identity map on $\Ker (\varepsilon)$.
It also follows that $d_1 d_2 =0$.

We wish to define $s_1$ so that $d_2 s_1 |_{\Ker (d_1)} = \mathbf{1}_{\Ker(d_1)}$ and
more generally so that 
\[
    d_2 s_1 + s_0d_1 = \mathbf{1}_{A\ot \k C_1} .
\]
First define $s_1$ on elements in $\Ker(d_1)$
by induction on their degrees,
starting with those that are least in the ordering, which are 
elements of $A\ot \k C_1$ corresponding to relations:
Ordering the elements of $C_2$ as $u_1,\ldots, u_\ell$, with $u_1$ least,
we define $s_1(d_2(1\ot u_1))= 1\ot u_1$.
Recall that we have chosen a total order on a vector space basis
of $A\ot \k C_1$, given by elements $r\ot v_i$ where $r$ is a normal
word and $v_i\in {\mathcal{B}}$, to coincide with the order on 
the corresponding words $rv_i$ in $T(V)$. 
Assume $s_1$ has been defined on elements of $\Ker (d_1)$
with highest term of degree (i.e.\ position in the total order) less
than $n$. 
Let $\sum_{j=1}^m a_{i_j} r_{i_j}\ot v_{i_j}\in \Ker (d_1)$ 
for some nonzero $a_{i_j}\in \k$, $r_{i_j}\in A$,
with $v_{i_1}, \ldots, v_{i_m}$ distinct elements of $\mathcal{B}$, 
and terms ordered so that $r_{i_1}\ot v_{i_1}$ is greatest and has degree~$n$.
Since $d_1(\sum a_{i_j}r_{i_j}\ot v_{i_j})=0$ in $A$ by assumption, 
and $r_{i_1}\ot v_{i_1}$ is greatest,
the  monomial $r_{i_1}v_{i_1}$ in $T(V)$ must contain a tip.
Since $r_{i_1}$ is a nonzero normal word, 
the tip must be a suffix (that is, right factor)
of $r_{i_1}v_{i_1}$, say $r_{i_1}v_{i_1} = v'u'$ in $T(V)$ with $u'$ a tip.
Since $u'$ is a tip, there is an element in the ideal $(I)$ of the form
$u' + \sum_{k=1}^\ell b_k t _k$ for some normal words $t _k$ and scalars $b_k$.
Write each $t_k = t_k ' v_{i_k}$ for  words $t_k'$. 
Set $\alpha = - \sum_{k=1}^\ell a_{i_1} b_k v' t_k '\ot v_{i_k}
+ \sum_{j=2}^m a_{i_j} r_{i_j}\ot v_{i_j}$ and 
\[
  s_1 \left( \sum_{j=1}^m a_{i_j} r_{i_j}\ot v_{i_j}\right) = a_{i_1} v'\ot u'  + s_1(\alpha) .
\]
Note that $\alpha$ consists of terms of lower degree than $r_{i_1}\ot v_{i_1}$,
and $\alpha \in\Ker (d_1)$ by construction, so 
we may now apply the induction hypothesis to define $s_1(\alpha)$. 
Recall that $A\ot \k C_1 = \Ker (d_1) \oplus \ima (s_0)$.  
We define $s_1$ on $\ima(s_0)$ to be 0.
We claim that by these definitions, 
\[
   d_2 s_1 + s_0 d_1 = \mathbf{1}_{A\ot \k C_1} .
\]
To see this, we check separately for elements of $\Ker(d_1)$ and
of $\ima (s_0)$. 
If $x=\sum_{j=1}^m a_{i_j}r_{i_j}\ot v_{i_j} \in \Ker (d_1)$ as above,
then by the inductive definition of $s_1$, we have
\[
   (d_2 s_1 + s_0 d_1)(x) = d_2 s_1(x) = x .
\]
If $x\in \ima (s_0)$ then 
\[
  (d_2 s_1 + s_0 d_1 )(x) = s_0 d_1(x) = x
\]
since $d_1 s_0 + \eta \varepsilon = {\mathbf{1}}_A$. 
It follows that $A\ot \k C_2  = \Ker (d_2)\oplus \ima (s_1)$:
If $b\in A\ot \k C_2$ then $b = (b-s_1d_2(b)) + s_1 d_2(b)$,
with $b-s_1d_2(b)$ in $\Ker (d_2)$ and $s_1 d_2 (b)$ in $\ima (s_1)$.
If $b\in \Ker (d_2)\cap \ima (s_1)$, write $b=s_1(c)$,
and we have $c= (d_2 s_1 + s_0 d_1)(c) = s_0 d_1 (c)$.
Then $b=s_1(c) = s_1 s_0d_1(c) = 0$ since $s_1 s_0 =0$ by definition of $s_1$. \vskip 7pt

{\bf Degree at least $3$:}
We take $n\geq 3$ and 
assume that $A$-module homomorphisms 
$d_1,\ldots, d_{n-1}$ and $\k$-linear maps $s_0,\ldots, s_{n-2}$
have been defined so that $d_{i-1}d_i =0$, $s_{i-1}s_{i-2} =0$, and 
$d_{i} s_{i-1} + s_{i-2} d_{i-1} = {\mathbf{1}}_{A\ot \k C_{i-1}}$ for $1\leq i\leq n-1$.
It follows by an argument similar to the above 
that
\[
  A\ot \k C_i = \Ker (d_i) \oplus \ima (s_{i-1})
\]
for all $1\leq i\leq n-1$. 
In particular, $A\ot \k C_{n-1} = \Ker (d_{n-1}) \oplus \ima (s_{n-2})$. 
We will define $d_n$ and $s_{n-1}$.
The map $d_n$ is defined first as follows. 
Let $u\in C_n$.
We may write uniquely $u = ru'$ for $u'\in C_{n-1}$ and $r$ in $\mathcal{R}$
by construction of the quiver $\mathbf{Q}$.
Let
\begin{equation}\label{eqn:dn-def}
   d_n(1\ot u) = r\ot u' - s_{n-2}d_{n-1}(r\ot u').
\end{equation}
Now $d_{n-1}(r\ot u') = r d_{n-1}(1\ot u')$ since $d_{n-1}$
is an $A$-module homomorphism, and in order to apply 
$s_{n-2}$ to this element of $A\ot \k C_{n-2}$, any elements of $A$
will need to be rewritten
as linear combinations of normal words before applying $s_{n-2}$ (since
$s_{n-2}$ is $\k$-linear but not an $A$-module homomorphism in general).
It follows directly from the definition of $d_n$ and the induction
hypothesis that $ d_{n-1} d_n =0$.

We wish to define $s_{n-1}$ so that $d_n s_{n-1}|_{\Ker(d_{n-1})} =
\mathbf{1}_{\Ker(d_{n-1})}$ and more generally so that 
\[
   d_n s_{n-1} + s_{n-2} d_{n-1} = \mathbf{1}_{A\ot \k C_{n-1}} . 
\]
The map $s_{n-1}$ is defined inductively as follows.
Let $\sum_{i=1}^m a_i r_i\ot u_i\in \Ker(d_{n-1})$ for some $a_i\in \k$,
normal words $r_i\in A$, and $u_i\in C_{n-1}$.
Recall that we have chosen a total order on a vector space basis
of $A\ot \k C_{n-1}$, given by elements $r\ot u$ where $r$ is a normal
word and $u\in  C_{n-1}$, to coincide with the order on the
corresponding words $ru$ in $T(V)$. 
We may assume $r_1\ot u_1$ is the highest term among all $r_i\ot u_i$.  
Write $u_1 = u'u''$, uniquely, where $u''\in C_{n-2}$.
Then by definition of $d_{n-1}$ (replacing $n$ by $n-1$
in equation~(\ref{eqn:dn-def})), we have 
\[
 0 =  d_{n-1}\left(\sum_{i=1}^m a_i r_i\ot u_i \right)  = a_1 r_1 u'\ot u'' + \beta ,
\]
where  
$\beta=-s_{n-3}d_{n-2}(a_1r_1u' \ot u'') + d_{n-1}(\sum_{i=2}^m a_ir_i\ot u_i)$,  
and when the term $d_{n-1} ( \sum a_i r_i \ot u_i)$ is
expanded, due to cancellation, the resulting expression for $\beta$
consists of terms lower in the order than $r_1\ot u_1$. 
Since $0 = a_1 r_1 u' \ot u'' + \beta$, 
considering $r_1u'$ as a word in $T(V)$,
there is a tip $v''$ that is a factor 
of $r_1u'$ in $T(V)$.
To make a unique choice of such a tip, 
write $r_1 = v_{j_1}\cdots v_{j_\ell}$ as a word in the letters in $\mathcal{B}$.
Now $u'$ is not a tip, but $r_1u'$ contains a tip, and so there is
a largest $k$ ($k\leq \ell$) for which $v_{j_k}\cdots v_{j_\ell} u'$ 
(uniquely) contains
a tip, and by construction this tip will then be a prefix.
Thus we may write, uniquely,
$r_1 u' = v'tu'$ where $t\in {\mathcal{R}}$ and $tu'$ uniquely
contains a tip that is a prefix.
So there is an arrow $u'\rightarrow t$ in the reduced
quiver $\overline{{\mathbf{Q}}}$ by definition.
Therefore $tu_1=tu'u''\in C_n$.
We may thus set
\begin{equation}\label{eqn:sn-1}
   s_{n-1} \left( \sum_{i=1}^m a_ir_i\ot u_i \right) = a_1 v'\ot
    tu_1 + s_{n-1}(\gamma) ,
\end{equation}
where 
\[
   \gamma = \sum_{i=1}^m a_i r_i \ot u_i  - d_n (a_1 v'\ot tu_1)
\]
has highest term that 
is lower in the order than $r_1\ot u_1$.
(To obtain the above expression, 
in the argument $\sum a_i r_i\ot u_i$ of $s_{n-1}$,
we have added and subtracted $d_n (a_1 v'\ot tu_1)$.)
Now continue in the same fashion to obtain
$s_{n-1}(\gamma)$ in terms of elements lower in the total order,
and so on.
Since the chosen basis of $A\ot \k C_{n-1}$ 
is well-ordered, we will eventually reach an
expression involving $s_{n-1}(0)=0$. 
 
As before, we define $s_{n-1}$ on $\ima (s_{n-2})$ to be 0,
so that $s_{n-1}s_{n-2}=0$. 
A calculation as before now shows that
\[
  d_n s_{n-1} + s_{n-2} d_{n-1} = \mathbf{1}_{A\ot \k C_{n-1}} .
\]
Now by its definition and the above arguments, 
$s_{\bu}$ is a contracting homotopy for the complex~(\ref{eqn:Anick}), implying
that the complex is exact, and therefore $(A\ot \k C_{\bu} , d_{\bu})$
is a free resolution of $\k$ as an $A$-module. 
\end{proof}


\subsection{A truncated polynomial ring}
\label{subsec:tp}

In this section, we let $\k$ be any field and $m_1,m_2,m_3$ be 
positive integers at least~2.
We look closely at the Anick resolution for the algebra
$A=\k [w,x,y]/(w^{m_1},x^{m_2},y^{m_3})$ which
will be used in the next section. 
For this purpose, 
choose generating set ${\mathcal{B}} = \{w, x,y\}$ and relations
\begin{equation}\label{eqn:3-relations}
I = \{w^{m_1}, \ x^{m_2}, \ y^{m_3}, wx-xw, \ wy-yw, \ xy-yx \} .
\end{equation}
Then $A$ has basis $\{ w^i x^j y^k \mid 0\leq i\leq m_1-1 , \
                                                 0\leq j\leq m_2-1 , \  
                                       0\leq k\leq m_3-1   \}$.
We choose the ordering $w<x < y $.
So, for example, the degree lex ordering on basis elements in degrees 0, 1, 2 is
\[
    1 < w < x < y < w^2 < wx < wy < x^2 < xy < y^2 .
\]
Note that $xy$ is a normal word, while $yx$ is a tip (or obstruction).
Generally, the normal words correspond to the  PBW basis of $A$, and 
the tips are
\[
  {\mathcal{T}}:=  \{ w^{m_1}, \ x^{m_2} , \ y^{m_3} , \ xw, \ yw, \ yx \} .
\]
The proper prefixes of the tips are 
\[
 {\mathcal{R}} = \{ w^i , \ x^j , \ y^k \mid 0\leq i\leq m_1-1 , \
                                                 0\leq j\leq m_2-1 , \ 
                                       0\leq k\leq m_3-1   \} .
\]

The corresponding reduced quiver $\overline{{\mathbf{Q}}}$, as defined in Section~\ref{subsec:Anick}, is as follows.
(The nonreduced quiver $\mathbf{Q}$ contains additional vertices and
arrows if $m>3$, but we will not need this here.)
If $m_1 =2$, the vertices $w$ and $w^{m_1 -1}$ are identified, and
there is a loop at that vertex. Similarly for $m_2$, $m_3$. 
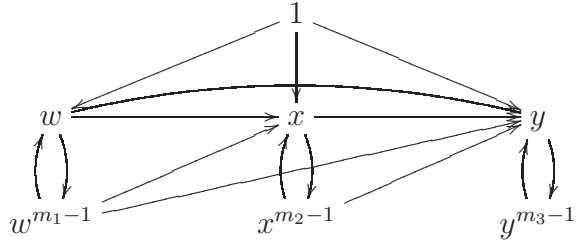
\[
\begin{xy}*!C\xybox{
\xymatrix{
      && 1\ar[dll]\ar[d]\ar[drr] && \\
   w\ar[rr]\ar@/^/[d]\ar@/^1pc/[rrrr] && x\ar[rr]\ar@/^/[d] && y\ar@/^/[d] \\
  w^{m_1-1}\ar@/^/[u]\ar[urr]\ar[urrrr] && x^{m_2-1} \ar@/^/[u]\ar[urr] && y^{m_3-1} \ar@/^/[u]
}}
\end{xy}
\]
We have 
\begin{eqnarray*}
  C_1 & = & \{ w, \ x , \ y \}, \\
  C_2 &=& \{ w^{m_1}, \ x^{m_2}, \ y^{m_3}, \ xw , \ yw , \  yx\}, \\
  C_3 & = & \{ w^{m_1+1}, \ x^{m_2+1}, \ y^{m_3+1}, xw^{m_1}, \ yw^{m_1}, \ yx^{m_2}, \ x^{m_2}w , \
     y^{m_3}w , \ y^{m_3}x , \ yxw \} , 
\end{eqnarray*}
and similarly we may find $C_n$ for $n>3$.

A free basis of $C_n$ is all $1\ot u$ where $u$ is a path of
length $n$ starting at 1 in the above reduced quiver $\overline{{\mathbf{Q}}}$. 
Fix such a free basis element $1\ot u$.
Suppose $n=i+j+k$ and the first $i$ vertices in the path $u$ are
in the set $\{w,w^{m_1-1}\}$, the second $j$ vertices of $u$ are in the set $\{x,x^{m_2-1}\}$,
and the third $k$ vertices are in the set $\{y,y^{m_3-1}\}$.
Write $u=u_{i,j,k}$ and note that the
triple of indices $i,j,k$ uniquely determines the path.
For convenience, we set $C_0 = \{ 1 \}$
and $u_{000} = 1$, identifying $A$ with $A\ot \k C_0$.
We set $u_{ijk} = 0$ if $i$, $j$, or $k$ is negative.

\begin{lemma}
Let $A=\k [w,x,y]/(w^{m_1},x^{m_2},y^{m_3})$, and let $P_{\bu}$ denote the Anick resolution
of $A$ with respect to the chosen generators $w,x,y$ and relations~(\ref{eqn:3-relations}). 
Then
\[
   d_{n} (1\ot u_{ijk}) = 
   y^{\sigma_3(k)} \ot u_{i, j, k-1} + (-1)^k x^{\sigma_2(j)} \ot u_{i, j-1, k}
    + (-1)^{j+k} w^{\sigma_1(i)} \ot u_{i-1, j , k} 
\]
where $n=i+j+k$, and $\sigma_a(\ell)=\begin{cases} 1 , & \text{ if } \ell \text{ is odd}, \\ 
                                                                             m_a-1 , & \text{ if } \ell \text{ is even}. \end{cases}$ 
\end{lemma}

\begin{proof}
We will prove the formula for $d_n$ by induction on $n$.
By definition, $d_1(1\ot u_{100})= w\ot u_{000}$,
$d_1(1\ot u_{010} ) = y \ot u_{000}$, and $d_1(1\ot u_{001})= x\ot u_{000}$,
and these values agree with the claimed formula for $d_1$.

Assume the formula holds for $d_{n-1}$.
We first consider the case $j=k=0$ and $i >0$: 
\begin{eqnarray*}
  d_n(1\ot u_{i00}) &=& w^{\sigma_1(i)}\ot u_{i-1,0,0} - s_{n-2} d_{n-1} (w^{\sigma_1(i)}\ot u_{i-1,0,0})\\
   &=& w^{\sigma_1(i)}\ot u_{i-1,0,0} - s_{n-2}(w^{\sigma_1(i)} (w^{\sigma_1(i-1)} \ot u_{i-2,0,0}))\\
   &=& w^{\sigma_1(i)}\ot u_{i-1,0,0} - s_{n-2}(0) \ \ = \ \ w^{\sigma_1(i)}\ot u_{i-1,0,0} ,
\end{eqnarray*}
since $ w^{\sigma_1(i)}w^{\sigma_1(i-1)} = 0$ in the algebra $A$.
This outcome agrees with the stated formula for $d_n$.
Next consider the case $k=0$ and $j>0$, applying the construction of $s_{n-2}$
described in the proof of Theorem~\ref{thm:Anick}:
The vertex $x^{\sigma_2(j)}$ is last in the path $u_{ij0}$, so by induction,  
\begin{eqnarray*}
   d_n(1\ot u_{ij0}) & = & x^{\sigma_2(j)} \ot u_{i,j-1,0} - s_{n-2} d_{n-1}(x^{\sigma_2(j)}
          \ot u_{i,j-1,0} ) \\
   &=&  x^{\sigma_2(j)} \ot u_{i,j-1,0} - s_{n-2} ( (-1)^{j-1} w^{\sigma_1(i)}
     x^{\sigma_2(j)} \ot u_{i-1,j-1,0}) \\
   &=&  x^{\sigma_2(j)}\ot u_{i,j-1,0} + (-1)^j w^{\sigma_1(i)} \ot u_{i-1,j,0} ,
\end{eqnarray*}
since $x^{\sigma_2(j)} x^{\sigma_2(j-1)} = 0$ in $A$.
This agrees with the stated formula for $d_n$.
In case $k>0$,
since the vertex labeled $y^{\sigma_3(k)}$ is the last in the path $u_{ijk}$, by
induction, since $y^{\sigma_3(k)}y^{\sigma_3(k-1)} = 0$, 
\[
\begin{aligned}
   & d_n(1\ot u_{ijk}) \\
  & = y^{\sigma_3(k)} \ot u_{i,j,k-1} - s_{n-2} d_{n-1}
    (y^{\sigma_3(k)} \ot u_{i,j,k-1}) \\
   &= y^{\sigma_3(k)} \ot u_{i,j,k-1} - s_{n-2}(y^{\sigma_3(k)} ( (-1)^{k-1} x^{\sigma_2(j)}\ot u_{i,j-1,k-1}
   +(-1)^{j+k-1} w^{\sigma_1(i)}\ot u_{i-1,j,k-1}))\\
  &= y^{\sigma_3(k)}\ot u_{i,j,k-1} - s_{n-2} ((-1)^{k-1}x^{\sigma_2(j)}y^{\sigma_3(k)}\ot u_{i,j-1,k-1
   } + (-1)^{j+k-1} w^{\sigma_1(i)}y^{\sigma_3(k)}\ot u_{i-1, j, k-1}) .
\end{aligned}
\]
Compare the two terms comprising the argument of $s_{n-2}$;
they are
$x^{\sigma_2(j)}y^{\sigma_3(k)}\ot u_{i,j-1,k-1}$
and $w^{\sigma_1(i)}y^{\sigma_3(k)}\ot u_{i-1, j, k-1}$, up to sign.
These terms have the same total degree, since each arises via an application
to $u_{i,j,k}$ of some differential maps (that do not change total degree).
Thus we must compare them lexicographically, and we see that 
$x^{\sigma_2(j)}y^{\sigma_3(k)}\ot u_{i,j-1,k-1}$ is the higher 
of the two terms. 
The first step of applying $s_{n-2}$ thus 
involves the term $x^{\sigma_2(j)}\ot u_{i,j-1,k}$ corresponding to this expression
as the first term on the right side of equation~(\ref{eqn:sn-1}).
Continuing by working inductively, with appropriate signs, we obtain
\[
    y^{\sigma_3(k)}\ot u_{i,j,k-1} + (-1)^k x^{\sigma_2(j)} \ot u_{i,j-1,k}
   + (-1)^{j+k} w^{\sigma_1(i)} \ot u_{i-1, j, k} ,
\]
as desired.
\end{proof}

Next we show that the Anick resolution is isomorphic to a twisted
tensor product resolution for this small example.

\begin{lemma}\label{lem:tp}
Let $A=\k[w,x,y]/(w^{m_1},x^{m_2},y^{m_3})$.
The Anick resolution $P_{\bu}$ for $A$ is equivalent to the total complex
$K_{\bu}$ of the tensor product of the minimal resolutions of
$A_w= \k[w]/(w^{m_1})$,
$A_x=\k[x]/(x^{m_2})$, and $A_y=\k[y]/(y^{m_3})$, that is,
for each $n$ there is an $A$-module isomorphism $\psi_n:P_n\rightarrow K_n$
and $\psi_{\bu}$ is a chain map lifting the identity map on $\k$.
\end{lemma} 

\begin{proof}
Let $P(A_w)$ be the following free resolution of $\k$ as an $A$-module:
\[
\begin{xy}*!C\xybox{\xymatrixcolsep{3pc}
\xymatrix{
   P(A_w): \quad \cdots\ar[r]^{\qquad (w^{m_1-1})\cdot} & A_w \ar[r]^{w\cdot}
                 & A_w \ar[r]^{(w^{m_1-1})\cdot}  & A_w \ar[r]^{w\cdot}
     & A_w \ar[r]^{\varepsilon} & \k\ar[r] & 0.}}
\end{xy}
\]
Let $P(A_x)$ and $P(A_y)$ be similar free resolutions of $\k$ 
as an $A_x$-module and as an $A_y$-module, respectively. 
Let $K_{\bu} = \mbox{Tot}(P(A_y)\ot P(A_x)\ot P(A_w))$, be the total complex of the tensor product
of these three complexes.
Let $\delta$ denote the differential on $P_{\bu}$.

We will show that $P_n\cong K_n$ as an $A$-module for each $n$
and that such isomorphisms may be chosen so as to constitute
a chain map between $P_{\bu}$ and $K_{\bu}$. 
We will prove this by induction on $n$, beginning with $n=0$ 
and $n=1$.
For $n=0$, note that $P_0= A \cong A_y\ot A_x\ot A_w = K_0$ and each maps onto $\k$ via $\varepsilon$.
We take $\psi_0$ to be this isomorphism. 

For $n=1$, note that $P_1 = A\ot \k \{w,x,y\}$, while $K_1$ is equal to 
\[
      (P(A_y)_1\ot P(A_x)_0\ot P(A_w)_0) 
          \oplus (P(A_y)_0\ot P(A_x)_1\ot P(A_w)_0)
          \oplus (P(A_y)_0\ot P(A_x)_0\ot P(A_w)_1 ) . 
\]
To keep track of degrees, let $\phi_{100}$ denote $1\ot 1\ot 1$ in
$P(A_y)_1\ot P(A_x)_0\ot P(A_w)_0$ and similarly $\phi_{010}$, $\phi_{001}$.
Let $\psi_1: P_1\rightarrow K_1$ be defined by
\[
   \psi_1(1\ot w) = \phi_{100}, \ \ \  \psi_1(1\ot x) = \phi_{010} \ \ \ \mbox{ and } 
   \ \ \ \psi_1(1\ot y) = \phi_{001} .
\]
More generally let $\phi_{ijk}$ denote $1\ot 1\ot 1$ in $P(A_y)_k\ot P(A_x)_j\ot P(A_w)_i$.
Recall similar notation $u_{ijk}$ for free basis elements of $P_n$
described above. 
Define $\psi_n: P_n\rightarrow K_n$ as follows:
\[
  \psi_n (1\ot u_{ijk}) =  \phi_{ijk} .
\]
Extend $\psi_n$ to an $A$-module isomorphism.

The differential on $K_{\bu}$ may be written as 
\[
    d_n (\phi_{ijk}) = y^{\sigma_3(k)}\phi_{i,j,k-1}
   + (-1)^k x^{\sigma_2(j)} \phi_{i,j-1,k} 
     + (-1)^{j+k} w^{\sigma_1(i)} \phi_{i-1,j,k} .
\]
Comparing with Lemma~\ref{lem:tp}, we see that  
$\psi_{\bu}$ is a chain map.
\end{proof}

\section{Finite generation of some cohomology rings}
\label{sec:fg}

We now apply the constructions of twisted tensor product and Anick resolutions discussed in Sections~\ref{sec:twisted} and \ref{sec:Anick} to prove that the cohomology rings of the Hopf algebras in our settings (see~Sections~\ref{subsec:settingR} and \ref{subsec:settingH}) are finitely generated.


\subsection{Cohomology of the Nichols algebra and its bosonization}
\label{subsec:cohomology grH}

Let $R,G$ be defined as in Section~\ref{subsec:settingR}. 
Recall that we have used a twisted tensor product to 
construct a resolution $K_{\bu}$ for $\k$ as an $R$-module in Section~\ref{subsec:resoln R} and further to construct a resolution $Y_{\bu}$ for $\k$ as a module over the bosonization $R\# \kG$ in Section~\ref{subsec:resoln bosonization}. 

By examining the expression for the cohomology 
$\coh^2(R\# \kG , \k)$ given at the end of Section~\ref{subsec:resoln bosonization}, 
we see that it includes nonzero elements represented by 
the 2-cocycles $\phi^*_{2,0,0}$, $\phi^*_{0,2,0}$, $\phi^*_{0,0,2}$.
We find their cup products, which will be used 
in the proof of Theorem~\ref{thm:main1} below. 

To simplify notation, let 
\[
   \xi_x = \phi_{2,0,0}^*, \qquad \xi_y=\phi^*_{0,2,0}, \qquad  \xi_g=\phi^*_{0,0,2}.
\]
Using the projectivity of the resolution $Y_{\bu}$, one can 
show that these functions may be extended to chain maps 
on $Y_{\bu}$ as follows: 
\[
    \xi_x (\phi_{i,j,k})  =  \phi_{i-2,j,k},  \qquad
   \xi_y (\phi_{i,j,k})  =  \phi_{i,j-2,k}, \qquad 
    \xi_g(\phi_{i,j,k})  =  \phi_{i,j,k-2}, 
\]
for all $i,j,k$, where we set $\phi_{i',j',k'}=0$ if any one of $i',j',k'$ is negative. Consequently, $\xi_x,\xi_y,\xi_g$ are generators of a polynomial subalgebra 
$\k[\xi_x,\xi_y,\xi_g]$ of $\coh^*(R\# \kG, \k)$ in even degrees.
For example, the above formulas can be used to show that
\[
   (\phi^*_{2,0,0})^2 = \phi^*_{4,0,0} \qquad \mbox{ and } \qquad 
   \phi_{2,0,0}^*\smile \phi^*_{0,2,0} = \phi^*_{2,2,0}=\phi^*_{0,2,0}\smile\phi^*_{2,0,0}
\]
and generally if $i,j,k, i',j',k'$ are all even, then
\[
    \phi^*_{i,j,k}\smile \phi^*_{i',j',k'} = \phi^*_{i+i',j+j',k+k'} .
\]

We will also need 
the following lemma, which is \cite[Lemma 2.5]{MPSW} as
adapted from~\cite[Lemma~1.6]{FS}. 

\begin{lemma}\label{lem:MPSW}
Let $E^{p,q}_1 \implies E^{p+q}_{\infty}$ be a multiplicative
spectral sequence of $\k$-algebras concentrated in the half
plane $p+q\geq 0$, and let $B^{*,*}$ be a bigraded commutative
$\k$-algebra concentrated in even (total) degrees.
Assume that there exists a bigraded map of algebras from $B^{*,*}$
to $E_1^{*,*}$ such that the image of $B^{*,*}$ consists
of permanent cycles, and $E_1^{*,*}$ is a noetherian module
over the image of $B^{*,*}$.
Then $E^*_{\infty}$ is a noetherian module over $\Tot (B^{*,*})$. 
\end{lemma}

We are now ready to prove our first main theorem. 

\begin{thm}\label{thm:main1}
Let $R := \k\langle x,y \rangle / (x^p, \ y^p , \ yx-xy-\frac{1}{2}x^2)$ be the Nichols algebra defined over a field $\k$ of prime characteristic $p>2$, and $G:=\langle g\rangle$ be a cyclic group of order $q$ divisible by $p$, 
acting on $R$ by automorphisms with ${}^g x = x$ and ${}^g y = x+y$.
Then the cohomology ring of the bosonization, 
$\coh^*(R \# \kG, \k)$, is finitely generated.
\end{thm}

\begin{proof}
Without loss of generality we may assume that $q = p^a$
for some $a$.
To see this, note that $G \cong \Z/ p^a \Z \times \Z/\ell \Z$ for
some $\ell$ coprime to $p$ and some $a \geq 1$.
Elements of the subgroup of $G$ that is isomorphic to $\Z/\ell \Z$ act trivially
on $R$ since their orders are coprime to $p$, and so $R\# \k G \cong
  (R\# \k \Z/ p^a\Z)\ot (\k \Z/\ell \Z )$ as an algebra.
Thus the cohomology of $R\# \k G$ is the graded tensor product of
the cohomology of $R\# \k \Z/p^a\Z$ and of $\k \Z/\ell \Z$.
The cohomology of $\k \Z/\ell\Z$ is concentrated in degree~0, where it
is simply $\k$, since $\ell$ is coprime to $p$. 

Now assume that $q = p^a$. 
Let $w=g-1$ and note that since the order of $g$ is $q$, 
\[
  R \# \kG\cong  \k\langle w,x,y\rangle / ( w^q, \ x^p, \ y^p, \ yx-xy-\frac{1}{2}x^2, \
     xw-wx, \ yw-wy-wx-x ) .
\]
Assign the degree lexicographic order on monomials in $w,x,y$,
with $w <x <y$.
This gives rise to an
$\N$-filtration on $R\# \k G$ for which 
the associated graded algebra 
$\gr (R\# \kG) \cong \k[w,x,y]/(w^q,x^p, y^p)$. 
(See, e.g.,~\cite[Theorem~4.6.5]{BGV}.)

We will apply the May spectral sequence \cite{M} in our context, for which:
$$E_1^{*,*} \cong \coh^*(\gr (R\# \kG),\k) \Longrightarrow E_{\infty}^{*,*} \cong \gr \coh^*(R\# \kG,\k).$$ 
The algebra $\gr (R\# \kG) \cong \k[w,x,y]/(w^q,x^p,y^p)$ 
has a  resolution  given by a
tensor product as in Lemma~\ref{lem:tp}, equivalently by 
repeating the twisted tensor product construction in 
Section~\ref{subsec:resoln bosonization} but with trivial twisting. 
We find that there are elements in degree $2$ of 
$\coh^*(\gr (R\# \kG),\k)$ corresponding to $\xi_w,\xi_x,\xi_y \in \coh^2(R\# \kG , \k)$ 
(here, we identify $\xi_w = \xi_g$), 
and we use the same notation for them, by abuse of notation. 
These elements are permanent cycles in the May spectral sequence: 
We have already seen that these cocycles exist for the filtered algebra $R\# \kG$, 
as constructed in Section~\ref{subsec:resoln bosonization}. 
They are permanent cocycles as we may identify their images with
the corresponding elements of $\coh^*(R\# \kG,\k)$.

Specifically, let $B^{*,*} = \k[\xi_w,\xi_x,\xi_y]$. By identifying
$\coh^*(\gr (R\# \kG),\k)$ with group cohomology, or by
arguments in \cite[Section 4]{MPSW}, we see that 
$E^{*,*}_1 \cong \coh^*(\gr (R\# \kG),\k)$ is a noetherian  $B^{*,*}$-module, 
(it is generated over $B^{*,*}$ by some elements $\eta_w,\eta_x,\eta_y$ in degree 1). 
By Lemma~\ref{lem:MPSW}, $E^{*,*}_{\infty} \cong \gr \coh^*(R\# \kG,\k)$ is a noetherian module over
$\k[\xi_w,\xi_x,\xi_y]$. By an appropriate Zariskian filtration \cite[Chapter 2]{HvO}, 
one can lift information from the associated graded ring to the filtered ring; 
thus, $\coh^*(R\# \kG ,\k)$ is noetherian over $\k[\xi_w,\xi_x,\xi_y]$. 
Therefore, by \cite[Proposition 2.4]{Evens}, 
$\coh^*(R\# \kG ,\k)$ is finitely generated as an algebra. 
\end{proof}

\begin{remark}{\em 
There is a different proof of Theorem~\ref{thm:main1} that
is closer to Evens' original proof of finite generation of group cohomology. 
See~\cite[Theorem~3.1, Section~5, and Erratum]{NWi} for details.  }
\end{remark}


\subsection{Cohomology of some pointed Hopf algebras of dimension~27} 
\label{subsec:cohomology lifting}

In this section, we let $\k$ be a field of characteristic~$p=3$ 
and consider the Hopf algebras $H(\epsilon,\mu,\tau)$ defined in Section~\ref{subsec:settingH}. 
Consider $\k$ to be the $H(\epsilon,\mu,\tau)$-module on which $w,x,y$ each act as $0$.

\begin{thm}\label{thm:main2}
Let $H(\epsilon,\mu,\tau)$ be a Hopf algebra of dimension~27 over an
algebraically closed 
field $\k$ of characteristic $p=3$, as defined in Section~\ref{subsec:settingH}.
The cohomology $\coh^*( H(\epsilon,\mu,\tau) , \k)$ is finitely generated. 
\end{thm}

\begin{proof}
Choose the ordering  $w<x<y$ as before, and the corresponding
degree lexicographic ordering on monomials.
Due to the form of the relations,  this gives rise to 
an $\N$-filtration on $H(\epsilon,\mu,\tau)$ for which 
the associated graded algebra is 
$\gr  H(\epsilon,\mu,\tau)\cong \k [w,x,y]/(w^3,x^3,y^3)$.
(See, e.g.,~\cite[Theorem~4.6.5]{BGV}.) 
We consider the Anick resolution $P_{\bu}$ for $H(\epsilon,\mu,\tau)$,
filtered correspondingly, and the resulting 
May spectral sequence of the complex 
$\Hom_{H(\epsilon,\mu,\tau)} (P_{\bu}, \k)$.
This is a multiplicative spectral sequence under the product induced by a diagonal map
$P_{\bu}\rightarrow P_{\bu}\ot P_{\bu}$ lifting the identity map on $\k$.
Note that the associated graded resolution to $P_{\bu}$ is $\gr P_{\bu}$,
which we may identify with 
the Anick resolution for $\gr H(\epsilon,\mu,\tau)$,
described in Lemma~\ref{lem:tp}. 

The Anick resolution for $A=H(\epsilon,\mu,\tau)$ has the same free basis
sets $C_n$ as that for $\gr H(\epsilon,\mu,\tau)$ described in Section~\ref{subsec:tp}.
Direct calculations show that it has 
the following differentials in degrees~2 and~3 (recall that the
parameter $\epsilon$ only
takes the values 0 or 1):
By the proof of Theorem~\ref{thm:Anick},  values of $d_2$ on tips
correspond to the relations, specifically, 
\begin{eqnarray*}
   d_2(1\ot w^3)&=& w^2\ot w , \\
   d_2(1\ot x^3) &=& x^2\ot x - \epsilon \ot x , \\
   d_2(1\ot y^3) &=& y^2\ot y + \epsilon y \ot y + (\mu\epsilon - \tau - \mu^2) \ot y , \\
d_2(1\ot xw) & = &  x\ot w - w\ot x -\epsilon w\ot w - \epsilon \ot w , \\
d_2(1\ot yw) &=& y\ot w - w\ot y - w\ot x - 1\ot x 
   + (\mu-\epsilon)w\ot w + (\mu-\epsilon) \ot w ,\\
d_2(1\ot yx) & = & y\ot x - x\ot y + x\ot x - (\mu + \epsilon)\ot x 
    -\epsilon \ot y + \tau w\ot w - \tau \ot w .
\end{eqnarray*}
Values of $d_3$ require some computation, using the algorithm
outlined as part of the proof of Theorem~\ref{thm:Anick}, and
on free basis elements they are: 
\begin{eqnarray*}
    d_3(1\ot w^4)& =& w \ot w^3 , \ \ \ 
    d_3(1\ot x^4) \ \ = \ \ x\ot x^3 , \ \ \ 
    d_3(1\ot y^4) \ \  = \ \ y\ot y^3 , \\
d_3(1\ot xw^3) & =& x\ot w^3 - w^2\ot xw , \\
d_3(1\ot x^3 w) &=& x^2\ot xw + w\ot x^3 + \epsilon wx\ot xw
     + \epsilon x\ot xw + \epsilon w\ot xw , \\
d_3(1\ot yw^3) & = & y\ot w^3 - w^2\ot yw + w^2 \ot xw + w\ot xw , \\
d_3 (1\ot yxw) & = & y\ot xw - x\ot yw + w\ot yx + \epsilon w\ot yw
     + x\ot xw + (\mu+\epsilon) w\ot xw , \\
d_3(1\ot y^3w) &=& y^2\ot yw + w\ot y^3 + w y \ot yx + wx\ot yx + (\epsilon-\mu) wy\ot yw \\
  && + (\mu-\epsilon)wx\ot yw -\tau w^2\ot yw + y\ot yx - (\epsilon +\mu) y\ot yw \\
  && + \tau w^2 \ot xw + x\ot yx + (\mu-\epsilon) x\ot yw 
   + (\mu^2 -\epsilon\mu) w \ot yw + \tau w\ot xw ,\\
d_3(1\ot yx^3) & = & y\ot x^3 - x^2\ot yx + \tau wx\ot xw 
    +\epsilon x\ot yx -\tau x\ot xw + \epsilon \tau w\ot xw ,\\
d_3(1\ot y^3x) &=& y^2\ot yx + x\ot y^3 - xy\ot yx -\tau wx \ot yw -\tau wy\ot yw \\
   && + \tau w^2\ot yx + \tau wx\ot xw + \epsilon\tau w^2\ot yw 
  + (\epsilon\tau +\mu\tau) w^2\ot xw \\ && + \mu y\ot yx 
   +\tau y\ot yw -\mu x\ot yx + \tau x\ot xw \\
  && + \tau w\ot yx + (\epsilon\tau + \mu\tau) w\ot yw +\epsilon\tau w\ot xw .
\end{eqnarray*} 
For example, to find $d_3 (1\ot yw^3   )$, 
we first compute
\[
  d_3(1\ot y w^3)  =  y\ot w^3 - s_1 d_2 (y\ot w^3) 
  = y\ot w^3 - s_1(yw^2\ot w) .
\]
Now, using the relations, rewrite $ y w^2$
as $w^2 y - w^2x - wx + (\mu+\epsilon) w^2 + \epsilon w $
so that the above expression is 
\[
  = y\ot w^3 - s_1 (w^2 y\ot w -  w^2 x\ot w -  w x\ot w 
    + (\mu+\epsilon ) w^2 \ot w + \epsilon w\ot w ) .
\]
Now $d_2 (w^2\ot yw) = w^2y\ot w - w^2\ot x + (\mu - \epsilon) w^2\ot w $ 
and, adding and subtracting $ - w^2 \ot x + (\mu-\epsilon) w^2\ot w$,
the above may be rewritten as
\[
\begin{aligned}
   & = y\ot w^3 - s_1( w^2 y \ot w - w^2\ot x + (\mu - \epsilon) w^2\ot w
     + w^2\ot x - (\mu - \epsilon) w^2\ot w \\
   & \quad -w^2x\ot w - wx\ot w + (\mu +\epsilon) w^2\ot w + \epsilon w\ot w)\\
  & = y\ot w^3 - w^2\ot yw - s_1(  w^2 \ot x -  w^2 x\ot w -  wx\ot w
    -\epsilon w^2\ot w + \epsilon w\ot w )) . 
\end{aligned}
\]
For the next two steps, we note that
$d_2(w^2\ot xw) = w^2x\ot w - \epsilon w^2\ot w$
and $d_2 (w\ot xw) = wx\ot w - w^2\ot x - \epsilon w^2 \ot w - \epsilon w\ot w$
and so the above may be rewritten as
\[
\begin{aligned}
  & = y\ot w^3 - w^2\ot yw - s_1(  w^2\ot x -  w^2 x\ot w + \epsilon
    w^2\ot w +\epsilon w^2\ot w  - w x\ot w +\epsilon w\ot w )\\
  &= y\ot w^3 - w^2\ot yw + w^2\ot xw 
   - s_1(w^2\ot x -  wx\ot w + \epsilon w^2\ot w
   + \epsilon w\ot w)  . 
\end{aligned}
\]
We recognize the argument of $s_1$ above as $d_2(w\ot xw)$ and so
we obtain
the value of $d_3(1\ot yw^3)$ as claimed. 

Looking at the values of $d_3$ given above, 
note that the coefficients in the factor $A = H(\epsilon,\mu,\tau)$ 
of $A\ot \k C_2$ in
the image of each of these free basis elements (under $d_3$)
are in the augmentation ideal.
(A similar statement does {\em not} apply to $d_2$.) 
In particular, letting $(w^3)^*, \ldots $ denote elements in the dual
basis in $\Hom_{\k} (\k C_2, \k)\cong \Hom_A ( A\ot \k C_2 , \k)$,
where $A = H(\epsilon,\mu,\tau)$, 
to the tips $w^3, \ldots $ of $\k C_2$, 
it follows that
\[
  d_3^*((w^3)^*)=0, \ \ \  d_3^*((x^3)^*) = 0 , \ \ \ d_3^*((y^3)^*)=0 .
\]
Setting $\xi_w = (w^3)^*$, $\xi_x = (x^3)^*$, $\xi_y = (y^3)^*$, we
see that these functions are cocycles in $\Hom_{\k} (\k C_2 , \k)
\cong \Hom_A(A\ot \k C_2, \k )$.

It follows from the above observations that
$\xi_w , \xi_x, \xi_y$ are permanent cocycles in the May spectral sequence,
and we may use them in an application of Lemma~\ref{lem:MPSW}:
On the $E_1$-page, $\xi_w,\xi_x,\xi_y$ correspond to analogous 2-cocycles
on $\gr H(\epsilon,\mu,\tau)$ that generate a polynomial
subalgebra of its cohomology ring by a similar analysis to that 
in earlier sections.
That is,  by Lemma~\ref{lem:tp}, the Anick resolution
is essentially the same as the (twisted) tensor product resolution used earlier.
Now let $B = \k [\xi_w,\xi_x,\xi_y]$.
Let $\eta_w = (w)^*$, $\eta_x = (x)^*$, $\eta_y = (y)^*$ in
$\Hom_{\k}(\k C_1, \k)\cong \Hom_A (A\ot \k C_1, \k)$. 
The cohomology of $\gr H(\epsilon,\mu,\tau)$ is finitely generated
as a module over $B$, by $\eta_w,\eta_x,\eta_y$ and their
products (note $\eta_w^2=0$, $\eta_x^2=0$, $\eta_y^2=0$,
so these products constitute a finite set).
By Lemma~\ref{lem:MPSW}, using an appropriate Zariskian 
filtration~\cite[Chapter~2]{HvO}, the cohomology  
$\coh^*(H(\epsilon,\mu,\tau), \k)$ is noetherian over 
$\k [\xi_w,\xi_x,\xi_y]$.
By~\cite[Propostion~2.4]{Evens}, it  is finitely 
as an algebra. 
\end{proof}

\begin{remark}{\em 
An alternative proof of our earlier Theorem~\ref{thm:main1} would proceed
just as our above proof of Theorem~\ref{thm:main2}:
One could compute the differentials on the Anick resolution of the
algebra there to show existence of the needed elements
$\xi_w,\xi_x,\xi_y$.
We chose instead to use the twisted tensor product construction,
for which we were able to give formulas for the differentials in all
degrees, yielding a more explicit, if not shorter, presentation. 
}
\end{remark}


\end{document}